%% file: m.tex
\documentclass[reqno,12pt]{amsart}

\include{environ}
\include{symbols}

\begin{document}

\title[A General Family of Regularized Navier-Stokes and MHD Models]
      {Analysis of a General Family of \\ Regularized Navier-Stokes and MHD Models}

\author[M. Holst]{Michael Holst}
\email{mholst@math.ucsd.edu}
\address{Department of Mathematics\\
         University of California San Diego\\
         La Jolla CA 92093}

\thanks{MH was supported in part by NSF Awards~0715146, 0411723,
and 0511766, and DOE Awards DE-FG02-05ER25707 and DE-FG02-04ER25620.}

\author[E. Lunasin]{Evelyn Lunasin}
\email{lunasin@math.arizona.edu}
\address{Department of Mathematics\\
         University of Arizona\\
         Tucson, AZ 85721}

\author[G. Tsogtgerel]{Gantumur Tsogtgerel}
\email{gantumur@math.mcgill.ca}
\address{Department of Mathematics and Statistics\\
         McGill University\\
         Montreal, Quebec H3A 2K6, Canada}

\thanks{EL and GT were supported in part by NSF Awards~0715146 and 0411723.}

\date{\today}


\input{abs}

\maketitle

\clearpage

\vspace*{-1.0cm}
{\footnotesize
\tableofcontents
}
\vspace*{-1.2cm}

\input{body}

\input{conc}

\section{Acknowledgements}
\input{ack}

\input{app}

\bibliographystyle{plain}
\bibliography{../bib/ref-gn,../bib/papers,../bib/books,../bib/mjh}

\vspace*{0.5cm}

\end{document}

%% file: environ.tex
\pdfoutput=1

\usepackage{color} % black,white,red,green,blue,cyan,magenta,yellow
\usepackage{graphicx}
\usepackage{pstricks}
\usepackage{ifpdf}
\ifpdf
    \usepackage[pdftex]{hyperref}
    \hypersetup{
        unicode=false,          % non-Latin characters in Acrobat’s bookmarks
        pdftoolbar=true,        % show Acrobat’s toolbar?
        pdfmenubar=true,        % show Acrobat’s menu?
        pdffitwindow=false,     % window fit to page when opened
        pdfstartview={FitH},    % fits the width of the page to the window
        pdftitle={MCP Article},      % title
        pdfauthor={Michael Holst},   % author
        pdfsubject={Mathematics},    % subject of the document
        pdfcreator={Michael Holst},  % creator of the document
        pdfproducer={Michael Holst}, % producer of the document
        pdfkeywords={PDE, analysis, mathematical physics}, % list of keywords
        pdfnewwindow=true,      % links in new window
        colorlinks=true,        % false: boxed links; true: colored links
        linkcolor=red,          % color of internal links
        citecolor=blue,         % color of links to bibliography
        filecolor=magenta,      % color of file links
        urlcolor=cyan           % color of external links
    }

\else
\fi

\usepackage{times}
\usepackage{amsfonts}

\usepackage{amsmath}
\usepackage{amsthm}
\usepackage{amssymb}
\usepackage{amsbsy}
\usepackage{amscd}

\usepackage{verbatim}
\usepackage{subfigure}

\newtheorem{theorem}{Theorem}[section]
\newtheorem{corollary}[theorem]{Corollary}
\newtheorem{lemma}[theorem]{Lemma}

\newtheorem{proposition}[theorem]{Proposition}

\newtheorem{definition}[theorem]{Definition}
\newtheorem{example}[theorem]{Example}
\newtheorem{remark}[theorem]{Remark}

\numberwithin{equation}{section}  %amsmath command: tie counter to section

  \newcounter{mnote}
  \let\oldmarginpar\marginpar
    \renewcommand\marginpar[1]{\-\oldmarginpar[\raggedleft\footnotesize #1]%
    {\raggedright\footnotesize #1}}
\definecolor{myblue}{rgb}{0.2,0.2,0.7}
\definecolor{mygreen}{rgb}{0,0.6,0}
\definecolor{mycyan}{rgb}{0,0.6,0.6}
\definecolor{myred}{rgb}{0.9,0.2,0.2}
\definecolor{mymagenta}{rgb}{0.9,0.2,0.9}
\definecolor{mywhite}{rgb}{1.0,1.0,1.0}
\definecolor{myblack}{rgb}{0.0,0.0,0.0}

%% file: symbols.tex
\newcommand*{\ang}[1]{\left\langle #1 \right\rangle}
\newcommand{\beq}{\begin{equation}}
\newcommand{\eeq}{\end{equation}}
\newcommand{\beqa}{\begin{eqnarray}}
\newcommand{\eeqa}{\end{eqnarray}}
\newcommand{\geqs}{\geqslant}      % Nice =<
      
\newcommand{\eps}{\varepsilon}

\newcommand{\N}{{\mathbb N}}       % Natural numbers.
\newcommand{\R}{{\mathbb R}}       % Real numbers.
\newcommand{\cS}{{\mathcal S}}

\def\Dd{\Delta}

\def\pp{\partial}
\def\veps{\varepsilon}

\newcommand*{\pref}[1]{(\ref{#1})}

\newcommand*{\paren}[1]{\left(#1\right)}

\newcommand*{\uns}{{\tilde{u}}}

\newcommand*{\ra}{{\rightarrow}}
\newcommand*{\tone}{{\theta_1}}
\newcommand*{\ttwo}{{\theta_2}}
\newcommand{\I}{I}          
\newcommand{\II}{I\!I}          
\newcommand{\III}{I\!I\!I}          

%% file: abs.tex
\begin{abstract}
We consider a general family of regularized Navier-Stokes and 
Magnetohydrodynamics (MHD) models 
on $n$-dimensional smooth compact Riemannian manifolds with or without
boundary, with $n \geqs 2$.
This family captures most of the specific regularized models that have been 
proposed and analyzed in the literature, including 
the Navier-Stokes equations, 
the Navier-Stokes-$\alpha$ model, 
the Leray-$\alpha$ model,
the Modified Leray-$\alpha$ model, 
the Simplified Bardina model,
the Navier-Stokes-Voight model, 
the Navier-Stokes-$\alpha$-like models,
and certain MHD models, 
in addition to representing a larger 3-parameter family of models not 
previously analyzed.
This family of models has become particularly important in
the development of mathematical and computational models of turbulence.
We give a unified analysis of the entire three-parameter family of models using 
only abstract mapping properties of the principal dissipation and smoothing
operators, and then use assumptions about the specific form of the 
parameterizations, leading to specific models, only when necessary to obtain 
the sharpest results.
We first establish existence and regularity results, and under
appropriate assumptions show uniqueness and stability.
We then establish some results for singular perturbations,
which as special cases include the inviscid limit of viscous models 
and the $\alpha \rightarrow 0$ limit in $\alpha$ models.
Next we show existence of a global attractor for the general model,
and then give estimates for the dimension of the global attractor and the 
number of degrees of freedom in terms of a generalized Grashof number.
We then establish some results on determining operators for the
two distinct subfamilies of dissipative and non-dissipative models.
We finish by deriving some new length-scale estimates in terms of
the Reynolds number, which allows for recasting the Grashof number-based
results into analogous statements involving the Reynolds number.
In addition to recovering most of the existing results on existence, 
regularity, uniqueness, stability, attractor existence and dimension, 
and determining operators for the well-known specific members of this 
family of regularized Navier-Stokes and MHD models, the framework we 
develop also makes possible a number of new results for all models in the 
general family, including some new results for several of the well-studied
models.
Analyzing the more abstract generalized model allows for a simpler analysis 
that helps bring out the core common structure of the various regularized 
Navier-Stokes and magnetohydrodynamics models, and also helps clarify the 
common features of many of the existing and new results.
To make the paper reasonably self-contained, we include supporting
material on spaces involving time,
Sobolev spaces, and Gronwall-type inequalities.
\end{abstract}

%% file: body.tex
%%%%%%%%%%%%%%%%%%%%%%%%%%%%%%%%%%%%%%%%%%%%%%%%%%%%%%%
\section{Introduction}
\label{s:intro}
%%%%%%%%%%%%%%%%%%%%%%%%%%%%%%%%%%%%%%%%%%%%%%%%%%%%%%%
The mathematical theory for global existence and regularity of solutions to the three-dimensional Navier-Stokes equations (3D NSE) is considered one of the most challenging unsolved mathematical problems of our time~\cite{Devlin02}.
It is also well-known that direct numerical simulation of NSE for many physical applications with high Reynolds number flows is intractable even using state-of-the-art numerical methods on the most advanced supercomputers available.
Over the last three decades, researchers have developed turbulence models as an attempt to side-step this simulation barrier; the aim of turbulence models is to capture the large, energetic eddies without having to compute the smallest dynamically relevant eddies, by instead modeling the effects of small eddies in terms of the large scales in both NSE and magnetohydrodynamics (MHD) flows.

In 1998, the globally well-posed 3D Navier-Stokes-$\alpha$ (NS-$\alpha$) equations (also known as the viscous Camassa-Holm equations and Lagrangian averaged Navier-Stokes-$\alpha$ model) was proposed as a sub-grid scale turbulence model~\cite{CFHOT98,CFHOT99a,CFHOT99b,FHT01,FHT02,HMR98,MaSh01a,MM03} (see also~\cite{ZhFa09b} for $n$-dimensional viscous Camassa-Holm equations in the whole space).
The inviscid and unforced version of 3D NS-$\alpha$ was introduced in~\cite{HMR98} based on Hamilton's variational principle subject to the incompressibility constraint $\mbox{div}\; u =0$ (see also~\cite{MaSh03a}).
By adding the correct viscous term and the forcing $f$ in an {\it ad hoc} fashion, the authors in~\cite{CFHOT98,CFHOT99a,CFHOT99b} and~\cite{FHT02} obtain the NS-$\alpha$ system.
In references~\cite{CFHOT98,CFHOT99a,CFHOT99b} it was found that the analytical steady state solutions for the 3D NS-$\alpha$ model compared well with averaged experimental data from turbulent flows in channels and pipes for a wide range of large Reynolds numbers.
It was this fact which led the authors of~\cite{CFHOT98,CFHOT99a,CFHOT99b} to suggest that the NS-$\alpha$ model be used as a closure model for the Reynolds-averaged equations.
Since then, it has been found that there is in fact a whole family of globally well-posed `$\alpha$'- models which yield similarly successful comparisons with empirical data; among these are the Clark-$\alpha$ model~\cite{CHT05}, the Leray-$\alpha$ model~\cite{CHOT05}, the modified Leray-$\alpha$ model~\cite{ILT06}, and the simplified Bardina model~\cite{CLT06, LL06}.

In addition to the early success of the $\alpha$-models mentioned above, the validity of the original $\alpha$-model (namely, the NS-$\alpha$ model) as a sub-grid scale turbulence model was also tested numerically in~\cite{CHMZ99,MM03}.
In the numerical simulation of the 3D NS-$\alpha$ model, the authors of~\cite{CHMZ99,GH99,GH03,MM03} showed that the large scales of motion bigger than $\alpha$ (the length scale associated with the width of the filter which regularizes the velocity field) in a turbulent flow are captured (see also~\cite{LKTT07,LKT08} for the 2D case and~\cite{CT09} for the rate of convergence of the 2D $\alpha$-models to NSE).
For scales of motion smaller than the length scale $\alpha$, the energy spectra decays faster in comparison to that of NSE.
This numerical observation has been justified analytically in~\cite{FHT01}.
In direct numerical simulation, the fast decay of the energy spectra for scales of motion smaller than the supplied filter length represents reduced grid requirements in simulating a flow.
The numerical study of~\cite{CHMZ99} gives the same results.
The same results hold as well in the study of the Leray-$\alpha$ model in~\cite{CHOT05,GH99}.

The NS-$\alpha$ turbulence model has also been implemented in a primitive equation ocean model (see~\cite{HHPW08a,HHPW08b,PHW08}).
Their simulations with the NS-$\alpha$ in an idealized channel domain was shown to produce statistics which resembles doubling of resolution.
For other applications of $\alpha$ regularization techniques, see~\cite{BLT07} for application to the quasi-geostrophic equations,~\cite{KhTi07} for application to Birkhoff-Rott approximation dynamics of vortex sheets of the 2D Euler equations, and~\cite{LiTi07,MMP05a,MMP05b} for applications to incompressible magnetohydrodynamics equations.
In~\cite{CMT08}, an $\alpha$-regularized nonlinear Schr{\"o}dinger equation 
%(called nonlinear Schr{\"o}dinger-Helmholtz equation) 
was proposed for the purpose of a numerical regularization that is hoped to shed some light on the profile of the blow-up solutions to the nonlinear Schr{\"o}dinger equations.
Also, in~\cite{BFMMW05}, the authors extend the derivation of the inviscid version of NS-$\alpha$ (called Euler-$\alpha$, also known as Lagrangian averaged Euler-$\alpha$) to barotropic compressible flows.

Perhaps the newest addition to the family of $\alpha$ turbulence model is the Navier-Stokes-Voight (NSV) equations proposed in~\cite{CLT06}, which turn out to also model the dynamics of Kelvin-Voight viscoelastic incompressible fluids as introduced in~\cite{Osk73,Osk80}.
The statistical properties of 3D NSV have been studied in~\cite{LRT09}.
The long-term asymptotic behavior of solutions is studied in~\cite{KLT07,KT07}.
In~\cite{EHL07a}, the NSV was used in the context of image inpainting.
The numerical study of NSV in~\cite{EHL07a} suggests that the NSV, in comparison with NSE, can provide a more efficient numerical process when automating the inpainting procedure for certain classes of images.
It is worthwhile to note that the inviscid NSV coincide with the inviscid simplified Bardina model which is shown in~\cite{CLT06} to be globally well-posed.
This new regularization technique for Euler equations prevents the risk of damping too much energy in the small scales which could lead to unrealistic numerical results.

As a representative of the more general model considered in this paper (described in detail in Section~\ref{s:prelim}), we consider first the following somewhat simpler constrained initial value problem on an 3-dimensional flat torus $\mathbb{T}^3$:
\begin{equation}
\begin{split}
\pp_t u + Au + (M u\cdot \nabla) (N u) + \chi \nabla (M u)^T\cdot (N u) 
  + \nabla p &= f(x),\\
\nabla\cdot u &= 0,\\
u(0) & = u_0,
\end{split}
\label{e:pde}
\end{equation}
where $A$, $M$, and $N$ are bounded linear operators having certain mapping properties, and where $\chi$ is either 1 or 0.
As in prior work on regularized models of the Navier-Stokes, and Euler equations, we employ a single real parameter $\theta$ to control the strength of the dissipation operator $A$.
We then introduce two parameters which control the degree of smoothing in the operators $M$ and $N$, namely $\theta_1$ and $\theta_2$, respectively, when $\chi = 0$, and $\theta_2$ and $\theta_1$, respectively, when $\chi = 1$.
Some examples of operators $A$, $M$, and $N$ which satisfy the mapping assumptions we will need in this paper are
\begin{equation}
A = (- \Delta)^{\theta},
\quad
M = (I - \alpha^2\Delta)^{-\theta_1},
\quad
N = (I - \alpha^2\Delta)^{-\theta_2},
\end{equation}
for fixed positive real number $\alpha$ and for specific choices of the real parameters $\theta$, $\theta_1$, and $\theta_2$.
However, we emphasize that the abstract mapping assumptions we employ are more general, and as a result do not require any specific form of the parameterizations of $A$, $M$, and $N$; this abstraction allows~\pref{e:pde} to recover most of the existing regularization models that have been previously studied, as well as to represent a much larger three-parameter family of models that have not been explicitly studied in detail.
As a result, the system in \pref{e:pde} includes the Navier-Stokes equations and the various previously studied $\alpha$ turbulence models as special cases, namely, the Navier-Stokes\,-\,$\alpha$ model~\cite{CFHOT98,CFHOT99a,CFHOT99b,FHT01,FHT02,HMR98,MM03}, Leray\,-\,$\alpha$ model~\cite{CHOT05}, modified Leray\,-\,$\alpha$ model (ML\,-\,$\alpha$)~\cite{ILT06}, simplified Bardina model (SBM)~\cite{CLT06}, Navier-Stokes-Voight (NSV) model~\cite{KLT07,KT07,Osk73,Osk80}, and the Navier-Stokes\,-\,$\alpha$\,-\,like (NS\,-\,$\alpha$\,-\,like) models~\cite{OlTi07}.
For clarity, some of the specific well-known regularization models recovered by~\pref{e:pde} for particular choices of the operators $A, M, N \mbox{ and } \chi$ are listed in Table~\ref{t:spec}.

\begin{comment}
\begin{table}[ht]
\caption{\footnotesize
Some special cases of the model \pref{e:pde} with $\alpha>0$, $\mathcal{S}=(I-\alpha^2\Delta)^{-1}$, 
and $\mathcal{S}_{\theta_2}~=~[I~+~(-\alpha^2\Delta)^{\theta_2}]^{-1}$.
}
\begin{center}
\begin{tabular}{|l|c|c|c|c|l|l|l|}
\hline
Model & $A$ & $M$ & $N$ & $\chi$ & $\theta$ & $\theta_1$ & $\theta_2$\\
\hline
NSE           & $-\nu\Delta$            & $I$        & $I$         & $0$ & & & \\
Leray-$\alpha$   & $-\nu\Delta$            & $\cS$ & $I$         & $0$ &&&\\
ML-$\alpha$      & $-\nu\Delta$            & $I$        & $\cS$  & $0$ &&&\\
SBM           & $-\nu\Delta$            & $\cS$ & $\cS$  & $0$ &&&\\
NSV           & $-\nu\Delta \cS$ & $\cS$   & $\cS$  & $0$ &&&\\
NS-$\alpha$      & $-\nu\Delta$            & $\cS$ & $I$         & $1$ &&&\\
NS-$\alpha$-like & $\nu(-\Delta)^{\theta}$ & $\cS_{\ttwo}$ & $I$ & $1$ &&&\\
\hline 
\end{tabular}
\end{center}
\label{t:spec}
\end{table}%
\end{comment}
%****************************************************************

{\small
\begin{table}[ht]
\caption{\footnotesize
Some special cases of the model \pref{e:pde} with $\alpha>0$, 
and with $\mathcal{S}=(I-\alpha\Delta)^{-1}$
and $\mathcal{S}_{\theta_2}~=~[I~+~(-\alpha\Delta)^{\theta_2}]^{-1}$.
%Each NSE model has a corresponding MHD analogue.
}

\begin{center}
\begin{tabular}{|c|c|c|c|c|c|c|c|} 
\hline\hline
Model  &  NSE      & Leray-$\alpha$ & ML-$\alpha$ & SBM        & NSV               & NS-$\alpha$       & NS-$\alpha$-like \\ \hline
 $A$   & $-\nu\Dd$ & $-\nu\Dd$      & $-\nu\Dd$   & $-\nu\Dd$  & $-\nu\Dd\cS$ & $-\nu\Dd$         & $\nu (-\Dd)^{\theta}$\\
 $M$   & $I$       & $\cS$     & $I$         & $\cS$ & $\cS$        & $\cS$        & $\cS_\ttwo$\\
 $N$   & $I$       & $I$            & $\cS$  & $\cS$ & $\cS$        & $I$               & $I$\\
$\chi$ & 0         &0               &0            & 0          &0                  &1                  &1\\
\hline\hline   
\end{tabular}
\end{center}
\label{t:spec}
\end{table}
}

Our main goal in this paper is to develop well-posedness and long-time dynamics results for the entire three-parameter family of models, and then subsequently recover the existing results of this type for the specific regularization models that have been previously studied.
Along these lines, we first establish a number of results for the entire three-parameter family, including results on existence, regularity, uniqueness, stability, linear and nonlinear perturbations (with the inviscid and $\alpha \rightarrow 0$ limits as special cases), existence and finite dimensionality of global attractors, and bounds on the number of determining degrees of freedom.
Elaborating on the latter a bit more, for $\theta>0$, we derived a lower bound for the number of degrees of freedom $m$ given by $m\geq G^{n/\theta}$, where $G$ is the Grashof number and $n$ is the spatial dimension.
A lower bound for the nondissipative case is also established.
These results give necessary and/or sufficient conditions on the ranges of the three parameters for dissipation and smoothing in order to obtain each result, and we indicate where appropriate which particular regularization models are covered in the allowable parameter ranges for each result.
In the final section of the paper, we develop some tools for relating the Grashof number-based results to analogous statements involving the Reynolds number.
Analyzing a generalized model based on abstract mapping properties of the principal operators $A$, $M$, and $N$ allows for a simpler analysis that helps bring out the core common structure of the various regularized NSE (as well as regularized magnetohydrodynamics) models, and also helps clarify the common features of many of the existing and new results.

In~\cite{OlTi07}, a two-parameter family of models was studied, corresponding to a subset of those studied here, which we will call here the NS-$\alpha$-like models.
In order to describe this subset of models, let $\theta$ and $\theta_2$ be two nonnegative parameters, and consider the following system on $\mathbb{T}^3$:
\begin{equation}
\begin{split}
\pp_t u + (-\Delta)^{\theta} u + (M u\cdot \nabla) u 
   +  \nabla (M u)^T\cdot u + \nabla p 
&= f, \\
\nabla\cdot u &= 0, \\
\nabla\cdot (Mu) &= 0,\\
u(0) &= u_0,
\end{split}
\label{e:olti}
\end{equation}
where $M=(I+(-\alpha^2\Delta)^{\ttwo})^{-1}$.
This family of NS-$\alpha$-like model equations interpolates between incompressible hyper-dissipative equations and the NS-$\alpha$ models when varying the two nonnegative parameters $\theta$ and $\theta_2$.
This is a special case of \pref{e:pde} with $\theta_1=0$ and $\chi=1$, with the degree of dissipation controlled by the parameter $\theta$ and the degree of nonlinearity controlled by only one parameter $\theta_2$.
In this particular case, the NSE are obtained when $\theta=1$ and $\theta_2=0$, while the NS-$\alpha$ model is obtained when $\theta=\theta_2=1$.
In~\cite{OlTi07}, sufficient conditions on the relationship between $\theta$ and $\theta_2$ are established to guarantee global well-posedness and global regularity of solutions.
Our result here can be viewed as generalizing the global well-posedness and regularity results in~\cite{OlTi07} to a larger three-parameter family using a more abstract framework that does not impose a specific form for the parameterizations, and then also establishing a number of additional new results for the larger three-parameter family, including results on stability, linear and nonlinear perturbations, existence and dimension of global attractors, and on determining operator bounds.

As a subset of the results mentioned above, we list some of the new results that we have obtained for the family of $\alpha$ models as a special case of the more generalized equation \pref{e:pde}.
As far as we know, these results have not been previously established in the literature.
The global existence and uniqueness of solutions for the {\it inviscid} $\alpha$ sub-grid scale turbulence models has been established only for the SBM~\cite{CLT06}.
Here, as a consequence of a more general result, we have established the global existence of a weak solution to the inviscid Leray-$\alpha$-model of turbulence.
In~\cite{FHT02}, the convergence of weak solutions of the NS-$\alpha$ to a weak solution of the NSE as the parameter $\alpha\ra 0$ was established.
Here we have established this convergence result as well for the NS-$\alpha$-like equations.
In addition, we have established for the NS-$\alpha$-like equations, the existence and finite dimensionality of its global attractor, and determining operator bounds.
In the case of Leray-$\alpha$, ML-$\alpha$, and SBM, the results on determining operator bounds also appear to be new.

%***************************************************************************

It is important to note that the general framework here allows for the development of results for certain (regularized or un-regularized) magnetohydrodynamics (MHD) models.
The basic MHD system has the form
\begin{equation}
\label{e:mhd}
\begin{split}
\pp_t u -\nu\Delta u + (u\cdot\nabla) u - (h\cdot\nabla) h
&= \nabla \pi - \textstyle\frac12\nabla|h|^2, \\
\pp_t h -\eta\Delta h + (u\cdot\nabla) h - (h\cdot\nabla) u    
&= 0,\\
\nabla\cdot h=\nabla\cdot u &= 0,
\end{split}
\end{equation}
where the unknowns are the velocity field $u$, the magnetic field $h$, and the pressure $\pi$, and where $\nu>0$ and $\eta>0$ denote the constant kinematic viscosity and constant magnetic diffusivity, respectively.
Our global well-posedness results include for example a particular regularized MHD model
\begin{equation}
\label{e:Leray-alpha-mhd}
\begin{split}
\pp_t u -\nu\Delta u + (Mu\cdot\nabla) u -(Mh\cdot\nabla) h 
&= \nabla \pi - \textstyle\frac12\nabla|h|^2, \\
\pp_t h -\eta\Delta h + (Mu\cdot\nabla) h - (Mh\cdot\nabla) u    
&= 0,
\end{split}
\end{equation}
supplemented with several divergence-free and boundary conditions, where $M=(I-\alpha^2\Delta)^{-1}$.
Note that the system \pref{e:Leray-alpha-mhd} is different from the 3D Leray-$\alpha$-MHD model proposed in~\cite{LiTi07}, where global well-posedness result is still open as in the case of the original MHD equations.
%%%\mnote{added in response to referee 2}
%%%{\bf
It is also different in nature from the MHD model proposed in~\cite{ZhFa09a}.
For the 3D Leray-$\alpha$-MHD model proposed in~\cite{LiTi07} the magnetic field $h$ is not regularized.
%%%}
Another regularized model whose global well-posedness result is covered here is the following modified version of the MHD-$\alpha$ model proposed in~\cite{LiTi07},
\begin{equation}
\label{e:mhd-alpha}
\begin{split}
\pp_t u -\nu\Delta u + (Mu\cdot\nabla) u + \nabla (Mu)^T\cdot u -(Mh\cdot\nabla) h 
&= \nabla \pi - \textstyle\frac12\nabla|h|^2, \\
\pp_t h -\eta\Delta h + (Mu\cdot\nabla) h - (Mh\cdot\nabla) u    
&= 0,
\end{split}
\end{equation}
supplemented with several divergence-free and boundary conditions.
Again, the above system is different from the original version of the MHD-$\alpha$ system proposed in~\cite{LiTi07}, which 
does not have a regularization on the magnetic field $h$.
For the original MHD-$\alpha$ system, global well-posedness is established in~\cite{LiTi07}.
Here we would like to stress that our current framework is best suited to MHD models where the velocity field $u$ and the magnetic field $h$
are treated on an equal footing as far as the regularization is concerned, so 
we cannot obtain sharpest results in our framework for the systems like the ones proposed in~\cite{LiTi07}.
However, our framework requires only minor modification to include these models; the function spaces become product spaces, and the principal dissipation and smoothing operators become block operators on these product spaces, typically with block diagonal form.
%%%\mnote{Just added}
%%%{\bf
It is worthwhile to note here that filtering the magnetic field, as is done in~\cite{MMP05a,MMP05b}, can be interpreted as introducing hyperviscosity for the filtered magnetic field.
This observation was first introduced in~\cite{LiTi07}, and smoothing the magnetic field was thought to yield unnecessary extra dissipation added to the system.
%%%} 

%%%\mnote{ Added in response to referee 1 with regards to ideal invariants}
%%%{\bf
Since the $\alpha$ models of turbulence were intended as a basis for regularizing numerical schemes for simulating turbulence, it is important to verify whether the ad~hoc smoothed equations inherit some of the original properties of the 3D MHD equations.
In particular, one would like to see if the ideal $(\nu=\eta=0)$ quadratic invariants of the smoothed MHD system can be identified with the ideal invariants of the original 3D MHD system under suitable boundary conditions.
There are three ideal quadratic invariants in 3D MHD (e.g.\ under rectangular periodic boundary conditions or in the whole space), namely, the energy $E = \frac{1}{2}\int_\Omega |u(x)|^2 + |h(x)|^2\;dx$, the cross helicity $h_c = \frac{1}{2}\int_\Omega u(x)\cdot h(x)\;dx$, and the magnetic helicity $h_M = \frac{1}{2}\int_\Omega a(x)\cdot h(x)\;dx$, where $a(x)$ is the vector potential so that $h(x) = \nabla\times a(x)$.
In the case of the 3D MHD-$\alpha$ equations in~\cite{LiTi07}, the three corresponding ideal invariants are the energy $E^\alpha~=~\frac{1}{2}\int_\Omega u(x)\cdot Mu(x)+|h(x)|^2\;dx$ (which reduce, as $\alpha\rightarrow 0$, to the conserved energy of the 3D MHD equations), the cross helicity $h^\alpha_c = \frac{1}{2}\int_\Omega u(x)\cdot h(x)\;dx$, and the magnetic helicity $h^\alpha_M = \frac{1}{2}\int_\Omega a(x)\cdot h(x)\;dx$.
For our system in \eqref{e:Leray-alpha-mhd} the corresponding ideal invariants are the energy $E^\alpha = \frac{1}{2}\int_\Omega |u(x)|^2 + |h(x)|^2\;dx$, and the  cross helicity $h_c^\alpha = \frac{1}{2}\int_\Omega u(x)\cdot h(x)\;dx$.
Currently, we are unable to identify a conserved quantity in the ideal version of \eqref{e:Leray-alpha-mhd} which correspond to the magnetic helicity in the 3D MHD system.
Similar problems arise in the 3D MHD-Leray-$\alpha$ equations introduced in~\cite{LiTi07}.
When both the magnetic field and the velocity field are filtered as it is done in~\cite{MMP05a}, the corresponding ideal quadratic invariants are $E^\alpha = \frac{1}{2}\int_\Omega u(x)\cdot Mu(x) + h(x)\cdot Mh(x)\;dx$, the cross helicity $h_c^\alpha = \frac{1}{2}\int_\Omega u(x)\cdot Mh(x)\;dx$, and the magnetic helicity $h^\alpha_M = \frac{1}{2}\int_\Omega Ma(x)\cdot Mh(x)\;dx$.
%%%}

%******************************************************************************
The remainder of the paper is structured as follows.
In~\S\ref{s:prelim}, we establish our notation and give some basic 
preliminary results for the operators appearing in the general 
regularized model.
In~\S\ref{s:well}, we build some well-posedness results for the general model; 
in particular, we establish existence results (\S\ref{ss:exist}), 
regularity results (\S\ref{ss:reg}),
and uniqueness and stability results (\S\ref{ss:stab}).
In~\S\ref{s:pert} we establish some results for singular perturbations,
which as special cases include the inviscid limit of viscous models and the 
$\alpha \rightarrow 0$ limit in $\alpha$ models;
this involves a separate analysis of the linear (\S\ref{ss:pert-lin}) and
nonlinear (\S\ref{ss:pert-nonlin}) terms.
These well-posedness and perturbation results are based on energy methods.
In~\S\ref{s:attr}, we show existence of a global attractor for the 
general model by dissipation arguments (\S\ref{ss:attr-exist}), 
and then by employing the classical approach from~\cite{Tema88}, give estimates for the dimension 
of the global attractor (\S\ref{ss:attr-dim}).
In~\S\ref{s:determining},
we establish some results on determining operators for the
two distinct subfamilies of dissipative (\S\ref{ss:dis})
and non-dissipative (\S\ref{ss:non-dis1}) models,
with the help of certain generalizations of the techniques used e.g., in~\cite{FMTT83,JoTi93,HoTi96b,KT07}.
Since the results in~\S\ref{s:determining} are (naturally) given in terms
of the generalized Grashof number, we finish in~\S\ref{s:Reynolds} by
developing some new results on length-scale estimates in terms of the
Reynolds number, which allows for relating the Grashof number-based results 
in the paper to analogous statements involving the Reynolds number.

To make the paper reasonably self-contained, 
in Appendix~\ref{s:app} we develop some supporting material on 
Gronwall-type inequalities (Appendix~\ref{ss:gronwall}),
spaces involving time (Appendix~\ref{ss:anis}),
and
Sobolev spaces (Appendix~\ref{ss:sobolev}).

%%%%%%%%%%%%%%%%%%%%%%%%%%%%%%%%%%%%%%%%%%%%%%%%%%%%%%%
\section{Notation and preliminary material}
\label{s:prelim}
%%%%%%%%%%%%%%%%%%%%%%%%%%%%%%%%%%%%%%%%%%%%%%%%%%%%%%%

Let $\Omega$ be an $n$-dimensional smooth compact manifold 
with or without boundary and equipped with a volume form,
and let $E\to\Omega$ be a vector bundle over $\Omega$ 
equipped with a Riemannian metric.
With $C^\infty(E)$ denoting the space of smooth sections of $E$,
let $\mathcal{V} \subseteq C^\infty(E)$ be a linear subspace,
let $A:\mathcal{V}\ra\mathcal{V}$ be a linear operator, 
and let $B:\mathcal{V}\times\mathcal{V}\ra\mathcal{V}$ be a bilinear map.
At this point $\mathcal{V}$ is conceived to be an arbitrary linear subspace of $C^\infty(E)$;
however, later on we will impose restrictions on $\mathcal{V}$ implicitly through various conditions on certain operators
such as $A$.
Assuming that the initial data $u_0\in\mathcal{V}$,
and forcing term $f\in C^\infty(0,T;\mathcal{V})$ with $T>0$,
consider the following equation
\begin{equation}\label{e:op}
\begin{split}
\pp_t u + Au + B(u,u) &= f,\\
u(0) &= u_0,
\end{split}
\end{equation}
on the time interval $[0,T]$.
Bearing in mind the model \pref{e:pde},
we are mainly interested in bilinear maps of the form 
\begin{equation}
   \label{e:b-def}
B(v,w)=\bar{B}(Mv,Nw),
\end{equation}
where $M$ and $N$ are linear operators in $\mathcal{V}$
that are in some sense regularizing and are relatively flexible,
and $\bar{B}$ is a bilinear map fixing the underlying nonlinear structure of the equation.
In the following, let $P:C^{\infty}(E)\ra\mathcal{V}$ be the $L^2$-orthogonal projector onto $\mathcal{V}$.

\begin{example}\label{x:spaces}
a) When $\Omega$ is a closed Riemannian manifold, and $E=T\Omega$ the tangent bundle,
an example of $\mathcal{V}$ is $\mathcal{V}_{\mathrm{per}}\subseteq \{u\in C^{\infty}(T\Omega):\mathrm{div}\,u=0\}$, 
a subspace of the divergence-free functions.
The space of periodic functions with vanishing mean on a torus $\mathbb{T}^n$
is a special case of this example.
In this case, one typically has $A = (- \Delta)^{\theta}$, $M = (I - \alpha^2 \Delta)^{-\theta_1}$, and $N = (I - \alpha^2 \Delta)^{-\theta_2}$.

b) When $\Omega$ is a compact Riemannian manifold with boundary, and again $E=T\Omega$ the tangent bundle,
a typical example of $\mathcal{V}$ is $\mathcal{V}_{\mathrm{hom}}=\{u\in C_0^{\infty}(T\Omega):\mathrm{div}\,u=0\}$ the space of compactly supported divergence-free functions.
In this case, we keep the operators 
$A = (- P\Delta)^{\theta}$, $M = (I - \alpha^2 P\Delta)^{-\theta_1}$, and $N = (I - \alpha^2 P\Delta)^{-\theta_2}$,
in mind as the operators that one would typically consider.

c) In either of the above two examples, one can consider as $\mathcal{V}$ the product spaces $\mathcal{V}_{\mathrm{per}}\times \mathcal{V}_{\mathrm{per}}$ and $\mathcal{V}_{\mathrm{hom}}\times \mathcal{V}_{\mathrm{hom}}$, which are encountered, e.g., in magnetohydrodynamics models, cf.
Example \ref{x:bilinear} below.
For the operators $A$, $M$, and $N$, we would have the corresponding block operators on the product space $\mathcal{V}$ for the above two examples, acting diagonally.
\end{example}

\begin{example}\label{x:bilinear}
a) In a) or b) of the Example~\ref{x:spaces} above, the bilinear map $\bar{B}$ can be taken to be 
\begin{equation}
\label{e:b1}
\bar{B}_1(v, w) = P [ (v\cdot\nabla) w ],
\end{equation}
which corresponds to the models with $\chi=0$ as discussed in~\S\ref{s:intro}.

b) Again in a) or b) of Example~\ref{x:spaces} above, $\bar{B}$ can be taken to be 
\begin{equation}
\label{e:b2}
\bar{B}_2(v, w) = P [ (w\cdot\nabla) v + (\nabla w^T)v ],
\end{equation}
which corresponds to the models with $\chi=1$ as discussed in~\S\ref{s:intro}.

c) An example of $B$ that cannot be written in the form~\pref{e:b-def} is the bilinear map for the Clark-$\alpha$ model,
which is 
\begin{equation*}
\label{e:bc}
B_c(v,w) = \bar{B}_1(Nv,w) + \bar{B}_1(w,Nv) - \bar{B}_1(Nv,Nw) - \alpha\bar{B}_1(\nabla^jNv,\nabla_jNw),
\end{equation*}
where $N=(I-\alpha^2 P\Delta)^{-1}$, and where the Einstein summation convention is assumed.
Note that for this bilinear map, one only has $\ang{B_c(v,v),Nv}=0$ for any $v\in\mathcal{V}_{\mathrm{per}}$
or $v\in\mathcal{V}_{\mathrm{hom}}$, in contrast to the examples $\bar{B}_1$ and $\bar{B}_2$
which are well-known to have stronger antisymmetry properties.

d) The MHD system \pref{e:mhd} has the bilinear map
\begin{equation}
\label{e:b3}
\bar{B}_3(v,w) 
= 
\left(\bar{B}_1(v_1,w_1) - \bar{B}_1(v_2,w_2) , \bar{B}_1(v_1,w_2) - \bar{B}_1(v_2,w_1) \right),
\end{equation}
where $v=(v_1,v_2)$ and $w=(w_1,w_2)$ are elements of either
$\mathcal{V}_{\mathrm{per}}\times \mathcal{V}_{\mathrm{per}}$ or $\mathcal{V}_{\mathrm{hom}}\times \mathcal{V}_{\mathrm{hom}}$.
The bilinear map for \pref{e:Leray-alpha-mhd} is a special case of $B_3(v,w)=\bar{B}_3(Mv,Nw)$,
with $M$ and $N$ having the form $M=\mathrm{diag}(M_1,M_1)$ and $N=I$.
Another special case of $B_3$ is the bilinear form for the Leray-$\alpha$-MHD model proposed in~\cite{LiTi07},
where $M$ and $N$ have the form  $M=\mathrm{diag}(M_1,I)$ and $N=I$.

e) The MHD-$\alpha$ model \pref{e:mhd-alpha} has the bilinear map of the form
\begin{equation}
\label{e:b4}
\bar{B}_4(v,w)= 
\left(\bar{B}_2(v_1,w_1) - \bar{B}_1(v_2,w_2) , \bar{B}_1(v_1,w_2) - \bar{B}_1(v_2,w_1)\right).
\end{equation}
The bilinear map for \pref{e:mhd-alpha} is a special case of $B_4(v,w)=\bar{B}_4(Mv,Nw)$,
with $M$ and $N$ having the form $M=\mathrm{diag}(M_1,M_1)$ and $N=I$.
Another special case of $B_4$ is the bilinear form for the MHD-$\alpha$ model proposed in~\cite{LiTi07},
where $M$ and $N$ have the form  $M=\mathrm{diag}(M_1,I)$ and $N=I$.

f) More generally, one can consider the bilinear maps of the form
\begin{equation*}
\label{e:b5}
\bar{B}_5(v,w)= \bar{B}_{i,j,k}(v,w) =
\left(\bar{B}_i(v_1,w_1) - \bar{B}_j(v_2,w_2) , \bar{B}_k(v_1,w_2) - \bar{B}_j(v_2,w_1)\right),
\end{equation*}
where $i,j,k\in\{1,2\}$.
This class includes the above examples d) and e).
\end{example}

To refer to the above examples later on,
let us introduce the shorthand notation:
\begin{equation}
\label{e:b-forms2}
B_i(v,w)=\bar{B}_i(Mv,Nw), \quad i=1,\ldots,5.
\end{equation}

As usual, we will study equation \pref{e:op} by extending it to function spaces that have weaker differentiability properties.
To this end, we interpret the equation \pref{e:op} in distribution sense,
and need to continuously extend $A$ and $B$ to appropriate smoothness spaces.
%As a byproduct, we now can allow the initial condition and the forcing term to be nonsmooth in certain sense.
Namely, we employ the spaces $V^s = \mathrm{clos}_{H^s}\mathcal{V}$, which will informally be called Sobolev spaces in the following.
The pair of spaces $V^s$ and $V^{-s}$ are equipped with the duality pairing $\ang{\cdot,\cdot}$, that is, 
the continuous extension of the $L^2$-inner product on $V^0$.
Moreover, we assume that there is a self-adjoint positive operator $\Lambda$
such that $\Lambda^s:V^{s}\to V^{0}$ is an isometry for any $s\in\R$.
For arbitrary real $s$, assume that $A$, $M$, and $N$ can be continuously extended so that
\begin{equation}\label{e:bdd-amn}
A:V^s \ra V^{s-2\theta},\quad
M: V^s \ra V^{s+2\theta_1},\quad\textrm{and}\quad
N: V^s \ra V^{s+2\theta_2},
\end{equation}
are bounded operators.
%As remarked earlier, examples are $A = (- P\Delta)^{\theta}$, $M = (I - \alpha\Delta)^{-\theta_1}$, and $N = (I - \alpha\Delta)^{-\theta_2}$, where the space $\mathcal{V}$ must be chosen so that the above operators are well behaved.
(Again, the assumptions we will need for $A$, $M$, and $N$ are more general, and do not require this particular form of the parameterization.)
We will assume $\theta\geq0$; however, there will not be any {\em a priori} sign restriction on $\theta_1$ and $\theta_2$.
We remark that $s$ in \pref{e:bdd-amn} is assumed to be arbitrary for the purposes of the discussion in this section, and that it is of course sufficient to assume \pref{e:bdd-amn} for a limited range of $s$ for most of the results in this paper.

\begin{remark}
Note that in this framework the best value for $\theta_1$ is $\theta_1=0$ for both the Leray-$\alpha$-MHD model and the MHD-$\alpha$ model as proposed in~\cite{LiTi07},
since those models have $M=(M_1,I)$, cf.\ Example \ref{x:bilinear} d) and e).
It is possible to refine our analysis by considering spaces such as $V^{s}\times V^{r}$ instead of $V^{s}\times V^{s}$.
\end{remark}

We assume that $A$ and $N$ are both self-adjoint, and coercive in the sense that for $\beta\in\R$,
\begin{equation}\label{e:coercive-a}
\ang{Av,\Lambda^{2\beta} v} \geq c_A\|v\|_{\theta+\beta}^2 - C_A\|v\|_{\beta}^2,
\qquad
v\in V^{\theta+\beta},
\end{equation}
with $c_A=c_A(\beta)>0$, and $C_A=C_A(\beta)\geq0$, and that
\begin{equation}\label{e:coercive-n}
\ang{Nv,v} \geq c_N\|v\|_{-\theta_2}^2,
\qquad
v\in V^{-\theta_2},
\end{equation}
with $c_N>0$.
We also assume that \pref{e:coercive-a} is valid for $\beta=-\theta_2$ with $\Lambda^{2\beta}$ replaced by $N$.
Note that if $\theta=0$, \pref{e:coercive-a} is strictly speaking not coercivity and follows from the boundedness of $A$,
and note also that \pref{e:coercive-n} implies the invertibility of $N$.
For clarity, we list in Table \ref{t:spec2} the corresponding values of the parameters and bilinear maps discussed above for special cases listed in Table \ref{t:spec}.

{\small
\begin{table}[ht]
\caption{\footnotesize
Values of the parameters $\theta, \theta_1$ and $\theta_2$, and the particular form of the bilinear map $B$ for some special cases of the model \pref{e:op}.
(The bilinear maps $B_1$ and $B_2$ are as in~\pref{e:b1}--\pref{e:b2},~\pref{e:b-forms2}.)
Each of the NSE models has a corresponding MHD analogue.
}
\begin{center}
\begin{tabular}{|c|c|c|c|c|c|c|c|} 
\hline\hline
Model  &  NSE      & Leray-$\alpha$ & ML-$\alpha$ & SBM        & NSV               & NS-$\alpha$       & NS-$\alpha$-like \\ \hline
$\theta$ &1        &1               &1            &1           &0                  &1                  &$\theta$\\
$\theta_1$ &0 &1 &0 &1 &1 &0 &0\\
$\theta_2$ &0 &0 &1 &1 &1 &1 &$\theta_2$\\
$B$ &$B_1$ &$B_1$ &$B_1$ &$B_1$ &$B_1$ &$B_2$ &$B_2$\\
\hline\hline   
\end{tabular}
\end{center}
\label{t:spec2}
\end{table}
}

We denote the trilinear form $b(u,v,w) = \langle B(u,v),w \rangle$,
and similarly the forms $\bar{b}$, $b_i$, and $\bar{b}_i$, $i=1,\ldots,5$,
following Example \ref{x:bilinear} and \pref{e:b-forms2}.
We consider the following {\em weak formulation} of the equation \pref{e:op}:
Find $u\in L^1_{\mathrm{loc}}(0,T;V^s)$ for some $s$ such that
\begin{equation}\label{e:weak}
\begin{split}
\hspace*{-0.2cm}
\int_0^T 
\hspace*{-0.2cm}
\big(
\hspace*{-0.1cm}
-
\hspace*{-0.1cm}
\ang{u(t),\dot{w}(t)} + \ang{A u(t), w(t)} 
   + b(u(t), u(t), w(t)) - \ang{f(t),w(t)} \big) dt &= 0
\\
u(0) &= u_0,
\end{split}
\end{equation}
for any $w\in C^\infty_0(0,T;\mathcal{V})$.
Here the dot over a variable denotes the time derivative.

The left hand side of the first equation in \pref{e:weak} is well defined if $u\in L^2(0,T;V^s)$, 
$f\in L^1(0,T;V^{s'})$, and
$b:V^{s}\times V^{s}\times V^{\gamma}\ra\R$ is bounded for some $s,s',\gamma\in\R$.
The second equation makes sense if $u\in C(0,\eps;V^{\bar{s}})$ for some $\eps>0$ and $\bar{s}\in\R$.
However, the following lemma shows that the latter condition is implied by the first equation.

\begin{lemma}
Let $\mathcal{X}\subset L^2(0,T;V^s)$ be the set of functions that satisfy the first equation in \pref{e:weak}.
Let $f\in L^1(0,T;V^{s-2\theta})$, and
let $b:V^{s}\times V^{s}\times V^{2\theta-s}\ra\R$ be bounded.
Then we have $\mathcal{X}\subset C(0,T;V^{s-2\theta})$.
\end{lemma}

\begin{proof}
If $u\in\mathcal{X}$, we have $\dot{u}\in L^1(0,T;V^{s-2\theta})$;
and so~\cite[Lemma 3.1.1]{Tema77} implies that $u\in C(0,T;V^{s-2\theta})$.
\end{proof}

In concluding the preliminary material in this section, we state the following result on the trilinear forms $\bar{b}_i$, $i=1,\ldots,5$, which are the main concrete examples for the trilinear form $\bar{b}$.
Recall that $b(u,v,w)=\bar{b}(Mu,Nv,w)$.

\begin{proposition}\label{p:bdd-b}
a) The trilinear forms $\bar{b}_i$, $i=1,\ldots,5$, are antisymmetric in their second and third variables:
\begin{equation}\label{e:b-anti}
\bar{b}_i(w,v,v) = 0,
\qquad
w,v\in\mathcal{V},
\quad
i=1,\ldots,5,
\end{equation}
where $\mathcal{V}$ is one of the appropriate spaces in Example \ref{x:spaces}.

b) The trilinear form $\bar{b}_1:V^{\sigma_1}\times V^{\sigma_2}\times V^{\sigma_3}\ra\R$ is bounded provided that
\begin{equation}\label{e:bdd-b}
\textstyle\sigma_1 + \sigma_2 + \sigma_3 > \frac{n+2}2,
\end{equation}
and that for some $k\in\{0,1\}$,
\begin{equation}\label{e:bdd-b1}
\sigma_2+\sigma_3\geq1,\quad
\sigma_1+\sigma_3\geq k,\quad
\sigma_1+\sigma_2\geq1-k.
\end{equation}
If the above three conditions are satisfied, and if $\sigma_i$ is a non-positive integer for some $i\in\{1,2,3\}$, then the inequality in \pref{e:bdd-b} can be replaced by the non-strict version of the inequality.
The non-strict inequality is also allowed if for some $k\in\{0,1\}$,
\begin{equation}\label{e:bdd-b1s}
\sigma_1\geq0,\quad
\sigma_2\geq k,\quad
\sigma_3\geq1-k.
\end{equation}

c) The trilinear form $\bar{b}_2:V^{\sigma_1}\times V^{\sigma_2}\times V^{\sigma_3}\ra\R$ is bounded provided that
\pref{e:bdd-b} holds
and in addition that
\begin{equation}\label{e:bdd-b2}
\sigma_2+\sigma_3\geq0,\quad
\sigma_1+\sigma_3\geq1,\quad
\sigma_1+\sigma_2\geq1.
\end{equation}
If the above three conditions are satisfied, and if $\sigma_i$ is a non-positive integer for some $i\in\{1,2,3\}$, then the inequality in \pref{e:bdd-b} can be replaced by the non-strict version of the inequality.
The non-strict inequality is also allowed if
\begin{equation}\label{e:bdd-b2s}
\sigma_1\geq1,\quad
\sigma_2\geq0,\quad
\sigma_3\geq0.
\end{equation}

d) The trilinear form $\bar{b}_3:V^{\sigma_1}\times V^{\sigma_2}\times V^{\sigma_3}\ra\R$ is bounded under the same conditions
on $\sigma_1$, $\sigma_2$, and $\sigma_3$ that are given in b).

e) The trilinear forms $\bar{b}_4, \bar{b}_5:V^{\sigma_1}\times V^{\sigma_2}\times V^{\sigma_3}\ra\R$ are bounded provided that
\pref{e:bdd-b} holds
and in addition that
\begin{equation}\label{e:bdd-b6}
\sigma_2+\sigma_3\geq1,\quad
\sigma_1+\sigma_3\geq1,\quad
\sigma_1+\sigma_2\geq1.
\end{equation}
If the above three conditions are satisfied, and if $\sigma_i$ is a non-positive integer for some $i\in\{1,2,3\}$, then the inequality in \pref{e:bdd-b} can be replaced by the non-strict version of the inequality.
The non-strict inequality is also allowed if for some $k\in\{0,1\}$,
\begin{equation}\label{e:bdd-b6s}
\sigma_1\geq1,\quad
\sigma_2\geq k,\quad
\sigma_3\geq1-k.
\end{equation}
\end{proposition}

\begin{proof}
The antisymmetricity of $\bar{b}_1$, and $\bar{b}_2$ is well known, and 
the boundedness of $B_1$ is immediate from Lemma~\ref{l:sob-hol}.
The antisymmetricity of  $\bar{b}_5$ (which {\em a fortiori} implies that of $\bar{b}_3$ and $\bar{b}_4$) can be seen from 
\begin{equation*}
\bar{b}_5(w,v,v) 
= 
\bar{b}_i(w_1,v_1,v_1) - \bar{b}_j(w_2,v_2,v_1) + \bar{b}_k(w_1,v_2,v_2) - \bar{b}_j(w_2,v_1,v_2),
\end{equation*}
where $i,j,k\in\{1,2\}$.
Part b) is proven by applying Lemma~\ref{l:sob-hol} to each of the following two representations
\begin{equation*}
\bar{b}_1(u,v,w)
= 
\ang{u^i\nabla_i v^k,w_k}
=
\ang{u^i v^k, \nabla_iw_k}.
\end{equation*}
Part c) is proven by applying Lemma~\ref{l:sob-hol} to
\begin{equation*}
\begin{split}
\bar{b}_2(u,v,w)
&= 
\ang{v^i \nabla_i u^k + u^i \nabla^k v_i , w_k}
=
\ang{v^i \nabla_i u^k + \nabla^k (u^i v_i) - v_i \nabla^k u_i , w_k}\\
&=
\ang{v^i \nabla_i u^k - v_i \nabla^k u_i , w_k}.
\end{split}
\end{equation*}
To complete the proof, d) and e) follow from parts b) and c).
\end{proof}

%%%%%%%%%%%%%%%%%%%%%%%%%%%%%%%%%%%%%%%%%%%%%%%%%%%%%%%
\section{Well-posedness results}
\label{s:well}
%%%%%%%%%%%%%%%%%%%%%%%%%%%%%%%%%%%%%%%%%%%%%%%%%%%%%%%

Similar to the Leray theory of NSE, we begin the development of a solution theory for the general 3-parameter family of regularized Navier-Stokes and MHD models with clear energy estimates that will be used to establish existence and regularity results, and under appropriate assumptions show uniqueness and stability.
To reinforce which existing results we recover in this general unified analysis, we give the corresponding simplified results that have been established previously in the literature for the special cases listed in Table \ref{t:spec}, at the end of the proof of every theorem.

%%%%%%%%%%%%%%%%%%%%%%%%%%%%%%%%%%%%%%%%%%%%%%%%%%%%%%%
\subsection{Existence}
\label{ss:exist}
%%%%%%%%%%%%%%%%%%%%%%%%%%%%%%%%%%%%%%%%%%%%%%%%%%%%%%%

In this subsection, we establish sufficient conditions for the existence of weak solutions to the problem \pref{e:weak}.
By a weak solution, we mean a solution satisfying
$u\in L^2(0,T;V^{\theta-\theta_2})$ and $\dot{u}\in L^1(0,T;V^{\gamma})$ for some $\gamma\in\R$.

\begin{theorem}\label{t:exist}
a) Let the following conditions hold.
\begin{itemize}
\item[i)] 
$b:V^{\sigma_1}\times V^{\sigma_2}\times V^{\gamma}\ra\R$ is bounded for some
$\sigma_i\in[-\theta_2,\theta-\theta_2]$, $i=1,2$, and $\gamma\in[\theta+\theta_2,\infty)\cap(\theta_2,\infty)$;
\item[ii)] 
$b(v,v,Nv)=0$ for any $v\in V^{\theta-\theta_2}$;
\item[iii)] 
$b:V^{\bar\sigma_1}\times V^{\bar\sigma_2}\times V^{\bar\gamma}\ra\R$ is bounded for some 
$\bar\sigma_i<\theta-\theta_2$, $i=1,2$, and $\bar\gamma\geq\gamma$;
\item[iv)] $u_0\in V^{-\theta_2}$, and $f\in L^2(0,T;V^{-\theta-\theta_2})$, $T>0$.
\end{itemize}
Then, there exists a solution $u\in L^\infty (0,T;V^{-\theta_2}) \cap L^2(0,T;V^{\theta-\theta_2})$ to \pref{e:weak}
satisfying
\begin{equation*}
\dot{u}\in L^p (0,T;V^{-\gamma}),
\quad
p=
\begin{cases}
\min\{2, \frac{2\theta}{\sigma_1 + \sigma_2 + 2\theta_2}\},\textrm{ if }\theta>0\\
2,\textrm{ if }\theta=0.
\end{cases}
\end{equation*}

b) With some $\beta\geq-\theta_2$, let the following conditions hold.
\begin{itemize}
\item[i)] 
$b:V^{\beta}\times V^{\beta}\times V^{\theta-\beta}\ra\R$ is bounded;
\item[ii)] 
$u_0\in V^{\beta}$, and $f\in L^2(0,T;V^{-\theta+\beta})$, $T>0$.
\item[iii)] 
$b:V^{\sigma}\times V^{\sigma}\times V^{\gamma}\ra\R$ is bounded for some
$\sigma<\theta+\beta$, and $\gamma\geq\theta-\beta$;
\end{itemize}
Then, there exist $T^*(u_0)=T^*(\|u_0\|_{\beta})>0$ and a local solution $u\in L^\infty (0,T^*;V^{\beta}) \cap L^2(0,T^*;V^{\theta+\beta})$ to the equation \pref{e:weak}.
\end{theorem}

\begin{remark}
a) Let $\theta+\theta_1>\frac12$.
Then from Proposition \ref{p:bdd-b} the trilinear forms $b_1$ and $b_3$ fulfill the hypotheses of a) for $-\gamma\leq\theta-\theta_2-1$ with 
$-\gamma<\min\{2\theta+2\theta_1-\frac{n+2}2,\theta-\theta_2+2\theta_1,\theta+\theta_2-1\}$.
Note that in particular this gives the global existence of a weak solution for the inviscid Leray-$\alpha$ model.
As far as we know this result has not been reported previously.

b) Let $\theta+2\theta_1\geq k$, $\theta+2\theta_2\geq1$, $\beta > \frac{n+2}{2}-\theta-2(\theta_1+\theta_2)$, and $\beta\geq \frac{1-k}{2} -\theta_1-\theta_2$, for some $k\in\{0,1\}$.
Then the trilinear forms $b_1$ and $b_3$ satisfy the hypothesis of b).
\end{remark}

\begin{remark}
a) Let $\theta+\theta_1>\frac12$.
Then the trilinear form $b_2$ fulfills the hypotheses of a) for $-\gamma\leq\theta-\theta_2-1$ with 
$-\gamma<\min\{2\theta+2\theta_1-\frac{n+2}2,\theta-\theta_2+2\theta_1-1,\theta+\theta_2\}$.

b) The trilinear form $b_2$ satisfies the hypothesis of b) for $\beta > \frac{n+2}{2}-\theta-2(\theta_1+\theta_2)$ with $\beta\geq \frac{1}{2} -\theta_1-\theta_2$ provided
$\theta+2\theta_1\geq1$ and $\theta+2\theta_2\geq0$.
\end{remark}

\begin{remark}
a) Let $\theta+\theta_1>\frac12$.
Then the trilinear forms $b_4$ and $b_5$ fulfill the hypotheses of a) for $-\gamma\leq\theta-\theta_2-1$ with 
$-\gamma<\min\{2\theta+2\theta_1-\frac{n+2}2,\theta-\theta_2+2\theta_1-1,\theta+\theta_2-1\}$.

b) Let $\theta+2\theta_1\geq 1$, $\theta+2\theta_2\geq1$, $\beta > \frac{n+2}{2}-\theta-2(\theta_1+\theta_2)$, and $\beta\geq \frac{1}{2} -\theta_1-\theta_2$.
Then the trilinear forms $b_1$ and $b_3$ satisfy the hypothesis of b).
\end{remark}

\begin{proof}[Proof of Theorem \ref{t:exist}]
Let $\{V_m:m\in\N\}\subset V^{\theta-\theta_2}$ be a sequence of finite dimensional subspaces of $V^{\theta-\theta_2}$ such that
\begin{enumerate}
\item
$V_m\subset V_{m+1}$ for all $m\in\N$;
\item
$\cup_{m\in\N}V_m$ is dense in $V^{\theta-\theta_2}$;
\item
For $m\in\N$, with $W_m=NV_m\subset V^{\theta+\theta_2}$,
the projector $P_m:V^{\theta-\theta_2}\ra V_m$ defined by
\begin{equation*}
\ang{P_mv,w_m}=\ang{v,w_m},
\qquad w_m\in W_m,\,v\in V,
\end{equation*}
is uniformly bounded as a map $V^{-\gamma}\ra V^{-\gamma}$.
\end{enumerate}
Such a sequence can be constructed e.g., by using the eigenfunctions of the isometry $\Lambda^{1+\theta}:V^{1+\theta-\theta_2}\to V^{-\theta_2}$.

Consider the problem of finding $u_m\in C^1(0,T;V_m)$ such that for all $w_m\in W_m$
\begin{equation}\label{e:galerkin}
\begin{split}
\langle\dot{u}_m , w_m \rangle + \langle A u_m , w_m \rangle + b(u_m, u_m, w_m) &= \langle f , w_m \rangle,\\
\langle u_m(0) , w_m \rangle &= \langle u_0 , w_m \rangle.
\end{split}
\end{equation}
Upon choosing a basis for $V_m$, the above becomes an initial value problem for a system of ODE's,
and moreover since $N$ is invertible by \pref{e:coercive-n}, the standard ODE theory gives a unique local-in-time solution.
Furthermore, this solution is global if its norm is finite at any finite time instance.

The second equality in \pref{e:galerkin} gives 
\begin{equation*}
c_N\|u_m(0)\|_{-\theta_2}^2
\leq \langle u_m(0) , N u_m(0) \rangle 
= \langle u(0) , N u_m(0) \rangle
\leq \|u(0)\|_{-\theta_2} \|Nu_m(0)\|_{\theta_2},
\end{equation*}
so that 
\begin{equation}
\|u_m(0)\|_{-\theta_2}\leq \frac{\|N\|_{-\theta_2;\theta_2}}{c_N}\|u(0)\|_{-\theta_2}.
\end{equation}
Now in the first equality of \pref{e:galerkin}, taking $w_m=Nu_m$, and using the condition ii) on $b$, we get
\begin{equation}\label{e:exist-1}
\begin{split}
\frac{d}{dt}\langle u_m, N u_m \rangle + 2 \langle Au_m, N u_m\rangle 
&= 2 \langle f_m, N u_m \rangle \\
&\leq \eps^{-1} \|f\|_{-\theta-\theta_2}^2 + \eps \|N\|_{-\theta_2;\theta_2}^2 \|u_m\|_{\theta-\theta_2}^2,
\end{split}
\end{equation}
for any $\eps>0$.
Since by choosing $\eps>0$ small enough we can ensure
\begin{equation*}
-2 \langle Au_m, N u_m\rangle + \eps \|N\|_{-\theta_2;\theta_2}^2 \|u_m\|_{\theta-\theta_2}^2
\lesssim
\|u_m\|_{-\theta_2}^2,
\end{equation*}
by Gr\"onwall's inequality we have
\begin{equation}\label{e:exist-2}
\|u_m (t)\|_{-\theta_2}^2
\lesssim
\left( \|u_m (0)\|_{-\theta_2}^2 + \int_{0}^t\|f\|_{-\theta-\theta_2}^2\right)
e^{C t},
\end{equation}
for some $C\in\R$.
For any fixed $T>0$, this gives  $u_m \in L^\infty (0,T;V^{-\theta_2})$ with uniformly (in $m$) bounded norm.
Moreover, integrating \pref{e:exist-1}, and taking into account \pref{e:exist-2}, we infer
\begin{equation}\label{e:exist-3}
\int_0^t \langle A u_m, N u_m \rangle \leq \psi(t),\qquad t\in[0,\infty),
\end{equation}
where $\psi:[0,\infty)\ra(0,\infty)$ is a continuous function.
If $\theta>0$, by the coerciveness of $A$, the above bound implies $u_m\in L^2(0,T;V^{\theta-\theta_2})$ with uniformly bounded norm.
So in any case, $u_m$ is uniformly bounded in $L^\infty (0,T;V^{-\theta_2})\cap L^2(0,T;V^{\theta-\theta_2})$,
and passing to a subsequence, there exists $u\in L^\infty (0,T;V^{-\theta_2})\cap L^2(0,T;V^{\theta-\theta_2})$ such that
\begin{equation}\label{e:exist-weak-conv}
\begin{split}
&u_m \rightarrow u \mbox{ weak-star in } L^\infty(0,T;V^{-\theta_2}),\\
&u_m \rightarrow u \mbox{ weakly in } L^2(0,T;V^{\theta-\theta_2}).
\end{split}
\end{equation}

For passing to the limit $m\ra\infty$ in \pref{e:galerkin} we shall need a strong convergence result,
which is obtained by a compactness argument.
We proceed by deriving a bound on $\dot{u}_m\equiv\frac{du_m}{dt}$.
Note that \pref{e:galerkin} can be written as
\begin{equation}\label{e:galerkin-abs}
\begin{split}
\dot{u}_m + P_m A u_m + P_m B(u_m, u_m) &= P_m f,\\
u_m(0) &= P_m u(0).
\end{split}
\end{equation}
Therefore
\begin{equation}
\| \dot{u}_m \|_{-\gamma} \leq C \| u_m \|_{\theta-\theta_2} + \| P_m B ( u_m , u_m ) \|_{-\gamma} + \| P_m f \|_{-\gamma} =: I_1 + I_2 + I_3.
\end{equation}
By the boundedness of $B$, we have
\begin{equation}
I_2 \lesssim \|u_m\|_{\sigma_1}\|u_m\|_{\sigma_2}.
\end{equation}
If $\theta=0$, then the norms in the right hand side are the $V^{-\theta_2}$-norm which is uniformly bounded.
If $\theta>0$, since
\begin{equation}
\|u_m\|_{\sigma_i} \leq \|u_m\|_{-\theta_2}^{1-\lambda_i}\|u_m\|_{\theta-\theta_2}^{\lambda_i},
\qquad
\lambda_i = \frac{\sigma_i + \theta_2}{\theta},
\quad
i=1,2,
\end{equation}
by the uniform boundedness of $u_m$ in $L^{\infty}(V^{-\theta_2})$ we have
\begin{equation}
I_2 \lesssim 
\|u_m\|_{-\theta_2}^{2-\lambda_1-\lambda_2}\|u_m\|_{\theta-\theta_2}^{\lambda_1+\lambda_2} 
\lesssim
\|u_m\|^{\lambda_1+\lambda_2}_{\theta-\theta_2}.
\end{equation}
Hence, with $\displaystyle\lambda := \lambda_1 + \lambda_2 = \frac{\sigma_1 + \sigma_2 + 2\theta_2}{\theta}$ if $\theta>0$,
and with $\lambda=1$ if $\theta=0$,
we get
\begin{equation}
\|\dot{u}_m\|_{L^p(V^{-\gamma})}^p 
\lesssim
\|u_m\|^p_{L^p(V^{\theta-\theta_2})} + \|u_m\|^p_{L^{p\lambda}(V^{\theta-\theta_2})} + \|f\|^p_{L^p(V^{-\theta-\theta_2})}.
\end{equation}
The first term on the right-hand side is bounded uniformly when $p\leq 2$.
The second term is bounded if $p\lambda\leq 2$,  that is $p\leq 2/\lambda$.
We conclude that $\dot{u}_m$ is uniformly bounded in $L^p\paren{0,T;V^{-\theta-\theta_2}}$, with $p=\min\{2,2/\lambda\}$.

By employing Theorem \ref{t:aubin},
we can now improve \pref{e:exist-weak-conv} as follows.
There exists $u\in C(0,T;V^{-\theta-\theta_2}) \cap L^\infty (0,T;V^{-\theta_2}) \cap L^2(0,T;V^{\theta-\theta_2})$ such that
\begin{equation}\label{e:exist-conv-u}
\begin{split}
&u_m \rightarrow u \mbox{ weak-star in } L^\infty(0,T;V^{-\theta_2}),\\
&u_m \rightarrow u \mbox{ weakly in } L^2(0,T;V^{\theta-\theta_2}),\\
&u_{m} \rightarrow u \mbox{ strongly in } L^2(0,T;V^{s}) \mbox{ for any }s<\theta-\theta_2.
\end{split}
\end{equation}
Now we will show that this limit $u$ indeed satisfies the equation \pref{e:weak}.
To this end,
let $w \in C^{\infty}(0,T;\mathcal{V})$ be an arbitrary function with $w(T)=0$, and let $w_m\in C^1(0,T;W_m)$ be such that $w_m(T)=0$ and $w_m\ra w$ in $C^1(0,T;V^{-\bar\gamma})$.
We have
\begin{multline}
- \int_0^T\langle u_m(t) , \dot{w}_m(t) \rangle dt 
+ \int_0^T \langle A u_m(t) , w_m(t) \rangle dt 
+ \int_0^T b( u_m(t) , u_m(t) , w_m(t) ) dt
\\
= \langle u_m(0) , w_m(0) \rangle
+  \int_0^T \langle f(t) , w_m(t) \rangle dt.
\end{multline}
We would like to show that each term in the above equation converges to the corresponding term in
\begin{multline}
- \int_0^T\langle u(t) , \dot{w}(t) \rangle dt 
+ \int_0^T \langle A u(t) , w(t) \rangle dt 
+ \int_0^T b( u(t) , u(t) , w(t) ) dt \\
= \langle u_0 , w(0) \rangle
+  \int_0^T \langle f(t) , w(t) \rangle dt,
\end{multline}
which would imply that $u$ satisfies \pref{e:weak}.
Here we show this only for the nonlinear term.
We have
\begin{equation}\label{e:b-conv}
\int_{0}^{T} \left\vert b( u_m(t) , u_m(t), w_m(t) ) - b( u (t) , u(t) , w(t) ) \right\vert dt
\leq \I_m + \II_m + \III_m,
\end{equation}
where the terms $\I_m$, $\II_m$, and $\III_m$ are defined below.
Firstly, it holds that
\begin{equation}
\begin{split}
\I_m 
&= \int_{0}^T | b( u_m(t) , u_m(t) , w_m(t)-w(t) ) | dt\\
&\lesssim
\int_{0}^T \|u_m(t)\|_{\bar\sigma_1} \|u_m(t)\|_{\bar\sigma_2} \| w_m(t) - w(t) \|_{\bar\gamma} dt\\
&\leq
\|u_m\|_{L^2(V^{\bar\sigma_1})} \|u_m\|_{L^2(V^{\bar\sigma_2})}  \|w_m-w\|_{C(V^{\bar\gamma})}.
\end{split}
\end{equation}
thus, we get $\lim_{m\rightarrow\infty}\I_m=0$ since {\em a fortiori} $\bar\sigma_i\leq\theta-\theta_2$, $i=1,2$.
For $\II_m$ we have
\begin{equation}
\begin{split}
\II_m 
&= 
\int_{0}^T | b( u_m(t)-u(t) , u_m(t) , w(t) ) | dt\\
&\lesssim 
\int_{0}^T \|u_m(t)-u(t)\|_{\sigma_1} \|u_m(t)\|_{\sigma_2} \|w(t)\|_{\theta+\theta_2} dt\\
&\leq 
\|u_m-u\|_{L^2(V^{\bar\sigma_1})} \|u_m\|_{L^2(V^{\bar\sigma_2})}  \|w\|_{C(V^{\bar\gamma})},
\end{split}
\end{equation}
so $\lim_{m\rightarrow\infty}\II_m=0$ since $\bar\sigma_1<\theta-\theta_2$ and $\bar\sigma_2\leq\theta-\theta_2$.
Similarly, we have
\begin{equation}
\begin{split}
\III_m 
&= 
\int_{0}^T | b( u(t) , u_m(t)-u(t) , w(t) ) | dt\\
&\lesssim 
\|u\|_{L^2(V^{\bar\sigma_1})} \|u_m-u\|_{L^2(V^{\bar\sigma_2})}  \|w\|_{C(V^{\bar\gamma})},
\end{split}
\end{equation}
so $\lim_{m\rightarrow\infty}\III_m=0$ since $\bar\sigma_1\leq\theta-\theta_2$ and $\bar\sigma_2<\theta-\theta_2$.

For the proof of b),
we choose the nested subspaces $\{V_m:m\in\N\}\subset V^{\theta+\beta}$ and $W_m=\Lambda^{2\beta}V_m\subset V^{\theta-\beta}$,
in \pref{e:galerkin}.
Then substituting $w_m=\Lambda^{2\beta}u_m(0)$ in the second equality in \pref{e:galerkin} gives
\begin{equation*}
\|u_m(0)\|_{\beta}
\lesssim 
\|u_0\|_{\beta}.
\end{equation*}
Now in the first equality of \pref{e:galerkin}, taking $w_m=\Lambda^{2\beta}u_m$, and using the boundedness of $b$, we get
\begin{equation*}
\frac{d}{dt}\|u_m\|_\beta^2 \lesssim \|f\|_{\beta-\theta}^2 + \|u_m\|_\beta^2 + \|u_m\|_\beta^4,
\end{equation*}
and thus
\begin{equation*}
\|u_m(t)\|_\beta \lesssim (T^*-t)^{-1/4},
\end{equation*}
for some $T^*>0$.
The rest of the proof proceeds similarly to that of a).
\end{proof}

For clarity, the corresponding conditions and results of Theorem \ref{t:exist} above are listed in Table \ref{t:spec-exist} below for the special standard model cases of the general three-parameter regularized model listed in Table \ref{t:spec}.
For the NS-$\alpha$-like case in the table, the allowed values for $\beta$ are
$\beta > \frac{5}{2}-\theta-2\theta_2$ with $\beta\geq \frac{1}{2} -\theta_2$ provided
$\theta\geq1$ and $\theta+2\theta_2\geq0$.

{\small
\begin{table}[ht]
\caption{\footnotesize
Existence results for some special cases of the model \pref{e:pde}.
The table gives values of $(a,b)$ and $(\gamma,p)$ for our recovered existence results for the standard models, giving existence of a solution $u\in L^\infty(V^a)\cap L^2(V^b)$, with $\dot{u}\in L^p(V^{-\gamma})$.
(In the NSV case, $\veps>0$, and in the NS-$\alpha$-like case, $\gamma \geq \max\{-\ttwo-1/2, \ttwo+1/2,n/2\}$.)
The last row indicates the local existence of solutions $u\in L^\infty(V^{\beta})\cap L^2(V^{\beta+\theta})$.
The result for each NSE model has a corresponding MHD analogue.
}

\begin{center}
\begin{tabular}{|c|c|c|c|c|c|c|c|} 
\hline\hline
Model  &  NSE      & Leray-$\alpha$ & ML-$\alpha$ & SBM        & NSV               & NS-$\alpha$       & NS-$\alpha$-like \\ \hline
 {\tiny Existence} & & & & & & & \\
{\tiny $a$, $b$} &0, 1 &0, 1 &-1, 0 &-1, 0 &-1, -1 &-1, 0 & ${-\ttwo}$, ${-\theta-\ttwo}$ \\
{\tiny $\gamma$, $p$} & 1, $\frac{4}{3}$  & 1, 2 &2, 2 &2, 2 & $1+\veps$, 2 & 2, 2 & $\gamma$, 2\\ 
%{\tiny $\dot{u}\in L^p(V^{-\gamma})$} & & & & &$\veps>0$ & & \\

%{\tiny $u_0\in $} &$V^0$    & $V^0$          & $V^{-1}$    & $V^{-1}$   &  $V^{-1}$         & $V^{-1}$          &$V^{-\ttwo}$ \\
%{\tiny $f\in L^2(X)$, $X=$} &$V^{-1}$ & $V^{-1}$& $V^{-2}$ & $V^{-2}$   &  $V^{-1}$         & $V^{-2}$          &$V^{-\theta-\ttwo}$ \\
{\tiny Local } & $\beta> \frac{3}{2}$& $\beta\geq 0$  &$\beta> -\frac{1}{2}$  &$\beta\geq -1$ &$\beta\geq -1$ & $\beta> -\frac{1}{2}$ & $\beta$  \\
\hline\hline   
\end{tabular}
\end{center}
\label{t:spec-exist}
\end{table}
}

%%%%%%%%%%%%%%%%%%%%%%%%%%%%%%%%%%%%%%%%%%%%%%%%%%%%%%%
\subsection{Uniqueness and stability}
\label{ss:stab}
%%%%%%%%%%%%%%%%%%%%%%%%%%%%%%%%%%%%%%%%%%%%%%%%%%%%%%%

Now we shall provide sufficient conditions for uniqueness and continuous dependence of on initial data for weak solutions of the general three-parameter family of regularized models.

\begin{theorem}\label{t:stab}
Let $\beta\geq-\theta_2$, and 
let $u_1,u_2\in L^{\infty}(0,T;V^{\beta}) \cap L^2(0,T;V^{\beta+\theta})$ be two solutions of \pref{e:weak} with initial conditions $u_1(0),u_2(0)\in V^{\beta}$, respectively.

a) Let $b:V^{\sigma_1}\times V^{\theta-\theta_2}\times V^{\sigma_2}\ra\R$ be bounded for some 
$\sigma_1\leq\theta-\theta_2$ and $\sigma_2\leq\theta+\theta_2$ with $\sigma_1+\sigma_2\leq\theta$.
Moreover, let $b(v,w,Nw)=0$ for any $v\in V^{\sigma_1}$ and $w\in V^{\sigma_2}$.
Then we have
\begin{equation}
\|u_1(t)-u_2(t)\|_{-\theta_2}\leq \phi(t)\|u_1(0)-u_2(0)\|_{-\theta_2},
\qquad 
t\in[0,T],
\end{equation}
where $\phi\in C([0,T])$ and $\phi(0)=1$.

b) Let $b:V^{\beta}\times V^{\beta}\times V^{\theta-\beta}\ra\R$ be bounded.
Then we have
\begin{equation}
\|u_1(t)-u_2(t)\|_{\beta}\leq \phi(t)\|u_1(0)-u_2(0)\|_{\beta},
\qquad 
t\in[0,T],
\end{equation}
where $\phi\in C([0,T])$ and $\phi(0)=1$.
\end{theorem}

\begin{remark}
The trilinear forms $b_1$ and $b_3$ satisfy the hypotheses of a) provided
$\theta+\theta_1\geq\frac{1-k}{2}$,
$\theta+2\theta_1\geq k$,
$\theta+\theta_2\geq\frac12$,
$2\theta+2\theta_1+\theta_2>\frac{n+2}2$, and
$3\theta+2\theta_1+2\theta_2\geq2-k$, for some $k\in\{0,1\}$.
The forms $b_1$ and $b_3$ satisfy the hypothesis of b) for
$\beta > \frac{n+2}{2}-2(\theta_1+\theta_2)-\theta$ with 
$\beta \geq \frac{1-k}{2}-(\theta_1+\theta_2)$ provided
$2\theta_2+\theta\geq 1$ and
$2\theta_1+\theta \geq k$, for some $k\in\{0,1\}$ .
\end{remark}

\begin{remark}
The trilinear form $b_2$ satisfies the hypotheses of a) for
$\theta+2\theta_1\geq1$,
$\theta+\theta_1\geq \frac{1}{2},$
$\theta+\theta_2\geq0$,
$2\theta+2\theta_1+\theta_2>\frac{n+2}2$, and
$3\theta+2\theta_1+2\theta_2\geq1$.
The trilinear form $b_2$ satisfies the hypothesis of b) for
$\beta > \frac{n+2}{2}-2(\theta_1+\theta_2)-\theta$ with 
$\beta \geq \frac{1}{2}-(\theta_1+\theta_2)$ provided
$2\theta_2+\theta\geq 0$ and
$2\theta_1+\theta \geq 1$.
\end{remark}

\begin{remark}
The trilinear forms $b_4$ and $b_5$ satisfy the hypotheses of a) provided
$\theta+\theta_1\geq\frac12$,
$\theta+2\theta_1\geq 1$,
$\theta+\theta_2\geq\frac12$,
$2\theta+2\theta_1+\theta_2>\frac{n+2}2$, and
$3\theta+2\theta_1+2\theta_2\geq2$.
The forms $b_4$ and $b_5$ satisfy the hypothesis of b) for
$\beta > \frac{n+2}{2}-2(\theta_1+\theta_2)-\theta$ with 
$\beta \geq \frac{1}{2}-(\theta_1+\theta_2)$ provided
$2\theta_2+\theta\geq 1$ and
$2\theta_1+\theta \geq 1$.
\end{remark}

\begin{proof}[Proof of Theorem \ref{t:stab}]
Let $v=u_1-u_2$.
Then subtracting the equations for $u_1$ and $u_2$ we have
\begin{equation}
\ang{\dot{v},w} + \ang{Av,w} + b(v, u_1, w) + b(u_2, v, w) =0.
\end{equation}
Taking $w=Nv$, we infer
\begin{equation*}
\begin{split}
\frac{d}{dt}\|v\|^2_{-\theta_2} + c\|v\|^2_{\theta-\theta_2}
& \leq C\|v\|_{\sigma_1}\|u_1\|_{\theta-\theta_2}\|v\|_{\sigma_2-2\theta_2}
\\
& \leq C\|v\|_{-\theta_2}^{2-\lambda_1-\lambda_2}\|v\|_{\theta-\theta_2}^{\lambda_1 + \lambda_2}\|u_1\|_{\theta-\theta_2},
\end{split}
\end{equation*}
where $\lambda_1=\frac{\sigma_1+\theta_2}{\theta}$ and $\lambda_2=\frac{\sigma_2-\theta_2}{\theta}$.
By applying Young's inequality we get
\begin{equation*}
\frac{d}{dt}\|v\|^2_{-\theta_2} \leq C\|v\|^2_{-\theta_2}\|u_1\|^2_{\theta-\theta_2}.
\end{equation*}
Now Gr{\"o}nwall's inequality gives
\begin{equation}
\|v(t)\|^2_{-\theta_2} \leq \|v(0)\|^2_{-\theta_2} \exp\int_0^tC\|u_1\|^2_{\theta-\theta_2}.
\end{equation}

The part b) is proven similarly, taking, e.g.\ $w=(I-\Delta)^{\beta} v$.
\end{proof}

To clarify these results in the case of specific models, the corresponding conditions and results of Theorem \ref{t:stab} above are listed in the Table \ref{t:spec-stab} below for the special case models listed in Table \ref{t:spec}.
{\small
\begin{table}[ht]
\caption{\footnotesize
Uniqueness results for some special cases of the model \pref{e:pde}.
The table gives values of $\beta$ for our recovered uniqueness results for the standard models, where $u_0\in V^\beta$ and where $u\in L^\infty (V^{\beta})\cap L^2(V^{\beta+\theta})$.
(In the NS-$\alpha$-like case, the requirement on $\beta$ is that $\beta > \max\{-\theta_2, 1/2-\theta_2,  5/2-\theta-2\theta_2 \}$).
The result for each NSE model has a corresponding MHD analogue.
}
\begin{center}
\begin{tabular}{|c|c|c|c|c|c|c|c|} 
\hline\hline
Model  &  NSE      & Leray-$\alpha$ & ML-$\alpha$ & SBM        & NSV               & NS-$\alpha$       & NS-$\alpha$-like \\ \hline

&&&&&&&\\
{\tiny Uniqueness} & $\beta> \frac{3}{2}$& $\beta\geq 0$  &$\beta> -\frac{1}{2}$  &$\beta\geq -1$ &$\beta\geq -1$ & $\beta> -\frac{1}{2}$ & $\beta$  \\
                 &                     &                & or $\beta=-1$         &               &               & or $\beta=-1$         &\\
\hline\hline   
\end{tabular}
\end{center}
\label{t:spec-stab}
\end{table}
}

%%%%%%%%%%%%%%%%%%%%%%%%%%%%%%%%%%%%%%%%%%%%%%%%%%%%%%%
\subsection{Regularity}
\label{ss:reg}
%%%%%%%%%%%%%%%%%%%%%%%%%%%%%%%%%%%%%%%%%%%%%%%%%%%%%%%

In this subsection, we develop a regularity result on weak solutions for the general family of regularized models.

\begin{theorem}\label{t:reg}
Let $u\in L^2(0,T;V^{\theta-\theta_2})$ be a solution to \pref{e:weak}, and
with some $\beta>-\theta_2$, let the following conditions hold.
\begin{itemize}
\item[i)] 
$b:V^{\alpha}\times V^{\alpha}\times V^{\theta-\beta}\ra\R$ is bounded, where $\alpha=\min\{\beta,\theta-\theta_2\}$;
\item[ii)] 
$b(v,w,Nw)=0$ for any $v,w\in \mathcal{V}$;
\item[iii)] 
$u(0)\in V^{\beta}$, and $f\in L^2(0,T;V^{\beta-\theta})$.
\end{itemize}
Then we have
\begin{equation}\label{e:reg}
u\in L^\infty(0,T;V^{\beta})\cap L^2(0,T;V^{\beta+\theta}).
\end{equation}
\end{theorem}

\begin{remark}
Let $4\theta+4\theta_1+2\theta_2>n+2$, 
$2\theta+2\theta_1\geq1-k$,
$\theta+2\theta_2\geq1$, 
$3\theta+4\theta_1\geq1$,
$\theta+2\theta_1\geq \ell$, 
and $3\theta+2\theta_1+2\theta_2\geq2-\ell$,
for some $k,\ell\in\{0,1\}$.
Let
$$\textstyle
\beta\in(\frac{n+2}{2}-2(\theta_1+\theta_2)-\theta, 3\theta+2\theta_1-\frac{n+2}2)\cap[\frac{1-\ell}{2}-\theta_1-\theta_2,\min\{2\theta+\theta_2-1,2\theta-\theta_2+2\theta_1-k\}].
$$
Then the trilinear forms $b_1$ and $b_3$ satisfy the hypotheses of the above theorem.
\end{remark}

\begin{remark}
Let 
$4\theta+4\theta_1+2\theta_2>n+2$, 
$\theta+2\theta_2\geq0$,
and
$\theta+2\theta_1\geq1$.
Let
$$\textstyle
\beta\in (\frac{n+2}{2}-2(\theta_1+\theta_2)-\theta, 3\theta+2\theta_1-\frac{n+2}2)
\cap
[\frac{1}{2}-\theta_1-\theta_2, \min\{2\theta+\theta_2,2\theta-\theta_2+2\theta_1-1\}].
$$
Then the trilinear form $b_2$ satisfies the hypotheses of the above theorem.
\end{remark}

\begin{remark}
Let $4\theta+4\theta_1+2\theta_2>n+2$, 
$\theta+2\theta_2\geq1$,
and
$\theta+2\theta_1\geq 1$.
Let 
$$\textstyle
\beta\in (\frac{n+2}{2}-2(\theta_1+\theta_2)-\theta, 3\theta+2\theta_1-\frac{n+2}2)
\cap
[ \frac{1}{2}-\theta_1-\theta_2 , \min\{2\theta+\theta_2-1,2\theta-\theta_2+2\theta_1-1\} ].
$$
Then the trilinear forms $b_4$ and $b_5$ satisfy the hypotheses of the above theorem.
\end{remark}

\begin{proof}[Proof of Theorem \ref{t:reg}]
By Theorems \ref{t:exist}.b) and \ref{t:stab}.a), there is $s>0$ depending on $\|u(0)\|_{\beta}$
such that $u\in L^\infty(0,s;V^{\beta})\cap L^2(0,s;V^{\beta+\theta})$.
With $I=[0,s)$, we have
\begin{equation}\label{e:innpro}
\ang{\dot{u},\Lambda^{2\beta}u} + \ang{Au,\Lambda^{2\beta}u} + b(u, u, \Lambda^{2\beta}u) =\ang{f,\Lambda^{2\beta}u},
\qquad a.e.\ \textrm{ in }I.
\end{equation}
By employing the boundedness of $b$ and the coercivity of $A$, we infer
\begin{equation*}
\frac{d}{dt}\|u\|_\beta^2 \lesssim \|f\|_{\beta-\theta}^2 +  \|u\|_{\theta-\theta_2}^2 \|u\|_{\beta}^2, 
\qquad a.e.\ \textrm{ in }I,
\end{equation*}
and using Gr\"onwall's inequality, we conclude
\begin{equation}
\|u(t)\|^2_\beta \lesssim \int_0^t \|f\|_{\beta-\theta}^2 + \|u(0)\|_\beta^2 \exp \int_0^t  C \|u\|^2_{\theta-\theta_2},
\qquad a.e.\ \textrm{ in }I,
\end{equation}
where the integral in the exponent is uniformly bounded since $u\in L^2(0,T;V^{\theta-\theta_2})$.
Therefore we have $u\in L^\infty(0,T;V^{\beta})$,
which transfers to $u\in L^2(0,T;V^{\theta+\beta})$ by the coercivity of $A$.
\end{proof}

Again for clarity, the corresponding conditions and results of Theorem \ref{t:reg} above are listed in the Table \ref{t:spec-reg} below for the special case models listed in Table \ref{t:spec}.
For the NS-$\alpha$-like case in the table, the allowed values for $\beta$ are $\beta\leq2\theta-\ttwo-1$ with $ \beta<3\theta-\frac{5}2$, provided that $\theta\geq\frac12$ and $4\theta+2\theta_2>5$.

{\small
\begin{table}[ht]
\caption{\footnotesize
Regularity results for some special cases of the model \pref{e:pde}.
The table gives values of $\beta$ for our recovered local and global regularity results for the standard models, where $u_0\in V^\beta$ and where $u\in L^\infty (V^{\beta})\cap L^2(V^{\beta+\theta})$.
(In the NS-$\alpha$-like case, see the text for the allowed values of $\beta$.)
Again, the result for each NSE model has a corresponding MHD analogue.
}

\begin{center}
\begin{tabular}{|c|c|c|c|c|c|c|c|} 
\hline\hline
Model  &  NSE      & Leray-$\alpha$ & ML-$\alpha$ & SBM        & NSV               & NS-$\alpha$       & NS-$\alpha$-like \\ %\hline
 {\tiny Regularity} &  & $\beta\leq 1$& $\beta\leq \frac12$  &$\beta\leq 2$  &$\beta\leq -\frac{1}{2}$  & $\beta\leq 0$ & $\beta$  \\
\hline\hline   
\end{tabular}
\end{center}
\label{t:spec-reg}
\end{table}
}

%%%%%%%%%%%%%%%%%%%%%%%%%%%%%%%%%%%%%%%%%%%%%%%%%%%%%%%
\section{Singular perturbations}
\label{s:pert}
%%%%%%%%%%%%%%%%%%%%%%%%%%%%%%%%%%%%%%%%%%%%%%%%%%%%%%%

In this section, we will consider the situation where the operators $A$ and $B$ in the general three-parameter family of regularized models represented by problem \pref{e:op} have values from a convergent (in a certain sense) sequence, and study the limiting behavior of the corresponding sequence of solutions.
As special cases we have inviscid limits in viscous equations and $\alpha\ra0$ limits in the $\alpha$-models.

%%%%%%%%%%%%%%%%%%%%%%%%%%%%%%%%%%%%%%%%%%%%%%%%%%%%%%%
\subsection{Perturbations to the linear part}
\label{ss:pert-lin}
%%%%%%%%%%%%%%%%%%%%%%%%%%%%%%%%%%%%%%%%%%%%%%%%%%%%%%%

Consider the problem
\begin{equation}\label{e:pert-lin-0}
\dot{u} + A u + B(u,u) = f,
\end{equation}
and its perturbation
\begin{equation}\label{e:pert-lin}
\dot{u}_i + A_i u_i + B(u_i,u_i) = f, 
\qquad
(i\in\N),
\end{equation}
where $A$, $B$, and $N$ (that will appear below) satisfy the assumptions stated in Section \ref{s:prelim},
and for $i\in\N$, $A_i:V^{s}\to V^{s-2\eps}$ is a bounded linear operator satisfying
\begin{equation}\label{e:pert-lin-ell}
\|A_iv\|_{-\eps-\theta_2}^2+\|v\|_{\theta-\theta_2}^2
%\lesssim\alpha_i\|v\|_{\eps-\theta_2}^2+\|v\|_{\theta-\theta_2}^2
\lesssim\ang{A_iv,Nv}+\|v\|_{-\theta_2}^2,
\quad
v\in V^{\eps-\theta_2}.
\end{equation}
Assuming that both problems \pref{e:pert-lin-0} and \pref{e:pert-lin} have the same initial condition
$u_0$,
and that $A_i\to A$ in some topology,
we are concerned with the behavior of $u_i$ as $i\to\infty$.
We will also assume that $\eps\geq\theta$.

\begin{theorem}\label{t:pert-lin}
Assume the above setting, and
in addition let the following conditions hold.
\begin{itemize}
\item[i)] 
$b:V^{\sigma_1}\times V^{\sigma_2}\times V^{\gamma}\ra\R$ is bounded for some
$\sigma_j\in[-\theta_2,\theta-\theta_2]$, $j=1,2$, and $\gamma\in[\eps+\theta_2,\infty)\cap(\theta_2,\infty)$;
\item[ii)] 
$b(v,v,Nv)=0$ for any $v\in\mathcal{V}$;
\item[iii)] 
$b:V^{\bar\sigma_1}\times V^{\bar\sigma_2}\times V^{\bar\gamma}\ra\R$ is bounded for some 
$\bar\sigma_j<\theta-\theta_2$, $j=1,2$, and $\bar\gamma\geq\gamma$;
\item[iv)] $u_0\in V^{-\theta_2}$, and $f\in L^2(0,T;V^{-\theta-\theta_2})$, $T>0$;
\item[v)] $A_i$ converge weakly to $A$ as $i\to\infty$.
\end{itemize}
Then, there exists a solution $u\in L^\infty (0,T;V^{-\theta_2}) \cap L^2(0,T;V^{\theta-\theta_2})$ to \pref{e:pert-lin-0}
such that up to a subsequence, 
\begin{equation}\label{e:pert-conv-u}
\begin{split}
&u_i \rightarrow u \mbox{ weak-star in } L^\infty(0,T;V^{-\theta_2}),\\
&u_{i} \rightarrow u \mbox{ weakly in } L^2(0,T;V^{\theta-\theta_2}), \\
&u_{i} \rightarrow u \mbox{ strongly in } L^2(0,T;V^{s}) \mbox{ for any }s<\theta-\theta_2,
\end{split}
\end{equation}
as $i\to\infty$.
\end{theorem}

\begin{proof}
Firstly, from Theorem~\ref{t:exist}, we know that 
for $i\in\N$
there exists a solution $u_i\in L^\infty (0,T;V^{-\theta_2}) \cap L^2(0,T;V^{\eps-\theta_2})$
to \pref{e:pert-lin}.
Duality pairing \pref{e:pert-lin} with $Nu_i$ and using elementary inequalities, we have
\begin{equation}\label{e:pert-1}
\frac{d}{dt}\ang{u_i, N u_i} + 2 \ang{A_iu_i, N u_i} = 2 \ang{f, Nu_i}\\
\lesssim \veps^{-1}\|f\|_{-\theta-\theta_2}^2 + \veps\|u_i\|_{\theta-\theta_2}^2.
\end{equation}
Choosing $\veps>0$ small enough, then using \pref{e:pert-lin-ell}, by Gr\"onwall's inequality we have
\begin{equation}\label{e:pert-2}
\|u_i (t)\|_{-\theta_2}^2
\lesssim
e^{Ct}.
\end{equation}
Moreover, integrating \pref{e:pert-1}, and taking into account \pref{e:pert-2} and \pref{e:pert-lin-ell}, we infer
\begin{equation}\label{e:pert-3}
\|A_iu_i\|_{L^2(0,t;V^{-\eps-\theta_2})}^2+\|u_i\|_{L^2(0,t;V^{\theta-\theta_2})}^2
\leq \psi(t),\qquad t\in[0,\infty),
\end{equation}
where $\psi:[0,\infty)\ra(0,\infty)$ is a continuous function.
For any fixed $T>0$, this gives  $u_i \in L^\infty (0,T;V^{-\theta_2})\cap L^2 (0,T;V^{\theta-\theta_2})$ with uniformly (in $i$) bounded norms.
On the other hand, we have
\begin{equation}
\|\dot{u}_i\|_{-\gamma} \leq  \|A_iu_i\|_{-\gamma} + \|B(u_i,u_i)\|_{-\gamma} + \|f\|_{-\gamma}.
\end{equation}
By estimating the second term in the right hand side as in the proof of Theorem \ref{t:exist}, 
and taking into account \pref{e:pert-3},
we conclude that $\dot{u}_i$ is uniformly bounded in $L^2(0,T;V^{-\gamma})$.
By employing Theorem \ref{t:aubin},
and passing to a subsequence,
we infer the existence of $u$ satisfying \pref{e:pert-conv-u}.
Now taking into account the weak convergence of $A_i$ to $A$,
the rest of the proof proceeds similarly to that of Theorem \ref{t:exist}.
\end{proof}

For example, setting $\veps=1$, with $\theta=0$ and $\theta_2=1$, and checking all the requirements $i)-v)$ of Theorem~\ref{t:pert-lin}, the viscous solutions to the SBM converge to the inviscid solution as the viscosity tends to zero.
Recall that the global existence of weak solution to the inviscid SBM (first established in~\cite{CLT06}) is also established in Theorem~\ref{t:exist}.
Similarly, setting $\veps=0$, with $\theta=0$ and $\theta_2=1$, the viscous solutions to the Leray-$\alpha$ model converge to the inviscid solution as the viscosity tends to zero.
This result gives another proof of the global existence of a weak solution for the inviscid Leray-$\alpha$ model.

On the other hand, the convergence of viscous solutions of ML-$\alpha$ and NS-$\alpha$ to its inviscid solutions, respectively, are not covered here since both models fail condition $i)$ of Theorem~\ref{t:pert-lin}.
Notice that the global existence of weak solution to the inviscid ML-$\alpha$ and NS-$\alpha$ are not established in Theorem~\ref{t:exist}.
Besides the inviscid SBM, there are no global well-posedness results reported previously in the literature for the other inviscid $\alpha$-models.

%%%%%%%%%%%%%%%%%%%%%%%%%%%%%%%%%%%%%%%%%%%%%%%%%%%%%%%
\subsection{Perturbations involving the nonlinear part}
\label{ss:pert-nonlin}
%%%%%%%%%%%%%%%%%%%%%%%%%%%%%%%%%%%%%%%%%%%%%%%%%%%%%%%

For $i\in\N$, let $A_i:V^{s}\to V^{s-2\eps}$ and $N_i:V^{s}\to V^{s+2\eps_2}$ be bounded linear operators,
satisfying
\begin{equation}\label{e:pert-nl-ell}
\|v\|_{\theta+\theta_2}^2
%\lesssim\alpha_i\|v\|_{\eps-\theta_2}^2+\|v\|_{\theta-\theta_2}^2
\lesssim\ang{A_iN_i^{-1}v,v}+\|v\|_{\theta_2}^2,
\qquad
v\in V^{\theta+\theta_2},
\end{equation}
and
\begin{equation}\label{e:pert-nl-nell}
\|v\|_{\theta_2}^2
%\lesssim\alpha_i\|v\|_{\eps-\theta_2}^2+\|v\|_{\theta-\theta_2}^2
\lesssim\ang{N_i^{-1}v,v},
\qquad
v\in V^{\theta_2},
\end{equation}
where we also assumed that $N_i$ is invertible.
In this subsection, we continue with perturbations of \pref{e:pert-lin-0}
of the form
\begin{equation}\label{e:pert-nonlin}
\dot{u}_i + A_i u_i + B_i(u_i,u_i) = f, 
\qquad
(i\in\N),
\end{equation}
where $B_i$ is some bilinear map.
Again assuming that both problems \pref{e:pert-lin-0} and \pref{e:pert-nonlin} have the same initial condition
$u_0$,
and that $A_i\to A$ and $B_i\to B$ in some topology,
we are concerned with the behavior of $u_i$ as $i\to\infty$.
For reference, define the trilinear form $b_i(u,v,w)=\ang{B_i(u,v),w}$.
%We will also assume that $\eps\geq\theta$.

\begin{theorem}\label{t:pert-nonlin}
Assume the above setting, and
in addition let the following conditions hold.
\begin{itemize}
\item[i)] 
$b_i:V^{\sigma}\times V^{\sigma}\times V^{\gamma}\ra\R$ is uniformly bounded for some
$\sigma\in[-\theta_2,\theta+\theta_2-2\eps_2)$, and $\gamma\in[\theta+\eps_2,\infty)\cap(\eps_2,\infty)$;
\item[ii)] 
$b_i(v,v,N_iv)=0$ for any $v\in\mathcal{V}$;
\item[iii)] 
$u_0\in V^{-\theta_2}$, and $f\in L^2(0,T;V^{-\theta-\theta_2})$, $T>0$;
\item[iv)] $A_i:V^{\theta-\theta_2}\to V^{-\gamma}$ is uniformly bounded and converges weakly to $A$;
\item[v)] $N_i^{-1}:V^{s+2\theta_2}\to V^{s+2\theta_2-2\eps_2}$ is uniformly bounded;
\item[vi)] $N_i^{-1}N:V^{\theta-\theta_2}\to V^{\theta+\theta_2-2\eps_2}$ converges strongly to the identity map;
\item[vii)] For any $v\in V^{\theta-\theta_2}$, $B_i(v,v)$ converges weakly to $B(v,v)$.
\end{itemize}
Then, there exists a solution $u\in L^\infty (0,T;V^{-\theta_2}) \cap L^2(0,T;V^{\theta-\theta_2})$ to \pref{e:pert-lin-0}
such that up to a subsequence, $y_i=N^{-1}N_iu_i$ satisfies
\begin{equation}\label{e:pertn-conv-u}
\begin{split}
&y_i \rightarrow u \mbox{ weak-star in } L^\infty(0,T;V^{-\theta_2}),\\
&y_{i} \rightarrow u \mbox{ weakly in } L^2(0,T;V^{\theta-\theta_2}), \\
&y_{i} \rightarrow u \mbox{ strongly in } L^2(0,T;V^{s}) \mbox{ for any }s<\theta-\theta_2,
\end{split}
\end{equation}
as $i\to\infty$.
\end{theorem}

\begin{proof}
Firstly, by Theorem \ref{t:exist}, we know that
for $i\in\N$
there exists a solution $u_i\in L^\infty (0,T;V^{-\eps_2}) \cap L^2(0,T;V^{\eps-\eps_2})$
to \pref{e:pert-nonlin}.
Pairing \pref{e:pert-nonlin} with $v_i:=N_iu_i$ and using elementary inequalities, we have
\begin{equation}\label{e:pertn-1}
\frac{d}{dt}\ang{N_i^{-1}v_i, v_i} + 2 \ang{A_iN^{-1}_iv_i, v_i} = 2 \ang{f, v_i}\\
\lesssim \veps^{-1}\|f\|_{\theta-\theta_2}^2 + \veps\|v_i\|_{\theta+\theta_2}^2.
\end{equation}
Choosing $\veps<0$ small enough, then using \pref{e:pert-nl-ell}, by Gr\"onwall's inequality 
and \pref{e:pert-nl-nell} we have
\begin{equation}\label{e:pertn-2}
\|v_i (t)\|_{\theta_2}^2
\lesssim
e^{Ct}.
\end{equation}
Moreover, integrating \pref{e:pertn-1}, and taking into account \pref{e:pertn-2} and \pref{e:pert-nl-ell}, we infer
that
for any fixed $T>0$, $v_i \in L^\infty (0,T;V^{\theta_2})\cap L^2 (0,T;V^{\theta+\theta_2})$ with uniformly (in $i$) bounded norms.
On the other hand, we have
\begin{equation}
\|\dot{u}_i\|_{-\gamma} \leq  \|A_iu_i\|_{-\gamma} + \|B_i(u_i,u_i)\|_{-\gamma} + \|f\|_{-\gamma}.
\end{equation}
By estimating the right hand side as in the proof of Theorem \ref{t:exist}, 
we conclude that $\dot{u}_i$ is uniformly bounded in $L^2(0,T;V^{-\gamma})$,
thus $\dot{v}_i=N_i\dot{u}_i$ is uniformly bounded in the same space.
By employing Theorem \ref{t:aubin},
and passing to a subsequence,
we infer the existence of $v\in L^\infty (0,T;V^{\theta_2}) \cap L^2(0,T;V^{\theta+\theta_2})$ satisfying 
\begin{equation}
\begin{split}
&v_i \rightarrow v \mbox{ weak-star in } L^\infty(0,T;V^{\theta_2}),\\
&v_{i} \rightarrow v \mbox{ weakly in } L^2(0,T;V^{\theta+\theta_2}), \\
&v_{i} \rightarrow v \mbox{ strongly in } L^2(0,T;V^{s}) \mbox{ for any }s<\theta+\theta_2,
\end{split}
\end{equation}
as $i\to\infty$.
Define $u=N^{-1}v$ and $y_i=N^{-1}v_i$, and note that these satisfy \pref{e:pertn-conv-u}.

Now we will show that this limit $u$ indeed satisfies the equation \pref{e:pert-lin-0}.
Let $w \in C^{\infty}(0,T;\mathcal{V})$ be an arbitrary function with $w(0)=w(T)=0$.
We have
\begin{multline*}
- \int_0^T\langle u_i(t) , \dot{w}(t) \rangle dt 
+ \int_0^T \langle A_i u_i(t) , w(t) \rangle dt 
+ \int_0^T b_i( u_i(t) , u_i(t) , w(t) ) dt \\
=  \int_0^T \langle f(t) , w(t) \rangle dt.
\end{multline*}
We claim that each term in the above equation converges to the corresponding term in
\begin{multline*}
- \int_0^T\langle u(t) , \dot{w}(t) \rangle dt 
+ \int_0^T \langle A u(t) , w(t) \rangle dt 
+ \int_0^T b( u(t) , u(t) , w(t) ) dt \\
= \int_0^T \langle f(t) , w(t) \rangle dt.
\end{multline*}
For the first term, we have
\begin{equation*}
\begin{split}
u_i-u
=
N_i^{-1}Ny_i-u
=
N_i^{-1}N(y_i-u)
+
(N_i^{-1}N-I)u,
\end{split}
\end{equation*}
and taking into account that $N_i^{-1}:V^{s+2\theta_2}\to V^{s+2\theta_2-2\eps_2}$ is uniformly bounded,
and that $N_i^{-1}N:V^{\theta-\theta_2}\to V^{\theta+\theta_2-2\eps_2}$ converges to the identity map
in the strong operator topology, 
we infer $u_i\to u$ in $L^2(0,T;V^{s+2\theta_2-2\eps_2})$ for any $s<\theta-\theta_2$.
For the second term, writing
\begin{equation*}
\begin{split}
A_iu_i-Au
=
A_i(u_i-u)
+
(A_i-A)u,
\end{split}
\end{equation*}
and taking into account the uniform boundedness of $A_i$, and the weak convergence $A_i\to A$,
prove the claim.
Finally, for the third term, we have
\begin{multline*}
B_i(u_i,u_i)-B(u,u)
=
B_i(u_i,u_i-u)
+
B_i(u_i-u,u)
+
B_i(u,u)-B(u,u),
\end{multline*}
and using the uniform boundedness of $B_i$ and the convergence of $B_i$ to $B$,
we complete the proof.
\end{proof}

For example, setting $\veps=\veps_2=1$, with $\theta=1$ and $\theta_2=0$, and checking all the requirements $i)-vii)$ of Theorem~\ref{t:pert-nonlin}, the weak solutions to the NS-$\alpha$ model converge to a weak solution of the NSE as the parameter $\alpha\rightarrow 0$.
This result was previously reported in~\cite{FHT02}.

%%%%%%%%%%%%%%%%%%%%%%%%%%%%%%%%%%%%%%%%%%%%%%%%%%%%%%%
\section{Global attractors}
\label{s:attr}
%%%%%%%%%%%%%%%%%%%%%%%%%%%%%%%%%%%%%%%%%%%%%%%%%%%%%%%

In this section we establish the existence of a global attractor for the general three-parameter family of regularized models, and give general requirements for estimating its dimension.
The dimension of the global attractor gives us some measure of the level of complexity of the dynamics of a given flow.
%That is, the general understanding (see, e.g.\cite{Temam88}) is that, the dimension of the global attractor can give us an estimate of the number of degrees of freedom of a turbulent phenomena.

%%%%%%%%%%%%%%%%%%%%%%%%%%%%%%%%%%%%%%%%%%%%%%%%%%%%%%%
\subsection{Existence of a global attractor}
\label{ss:attr-exist}
%%%%%%%%%%%%%%%%%%%%%%%%%%%%%%%%%%%%%%%%%%%%%%%%%%%%%%%

The following theorem establishes the existence of an absorbing ball in $V^{-\theta_2}$.
Moreover, with additional conditions, it shows not only the existence of an absorbing ball in a higher smoothness space
$V^{\beta}$, but also that any solution with initial condition in $V^{-\theta_2}$
acquires additional smoothness in an infinitesimal time, in particular implying that the absorbing ball in $V^{-\theta_2}$ is compact.

\begin{theorem}\label{t:attr-exist}
a) Let $u\in L^{\infty}_{\mathrm{loc}}(0,\infty;V^{-\theta_2}) \cap L^{2}_{\mathrm{loc}}(0,\infty;V^{\theta-\theta_2})$
be a solution to \pref{e:weak} with $u(0)\in V^{-\theta_2}$.
In addition, let the following conditions hold.
\begin{itemize}
\item[(i)]
$\ang{Av,Nv}\geq c \|v\|_{\theta-\theta_2}^2$ for any $v\in V^{\theta-\theta_2}$, with a constant $c>0$.
\item[(ii)]
$\displaystyle\sup_{t\geq0}\|f\|_{L^2(t,t+T;V^{-\theta-\theta_2})}^2\leq K$, where $T>0$ and $K\geq0$ are constants;
\end{itemize}
Then for some constant $k>0$ and for any $T'\geq0$, we have 
\begin{equation}\label{energy_f}
\|u(t)\|_{-\theta_2}^2
+ \|u\|_{L^2(t,t+T';V^{\theta-\theta_2})}^2
\lesssim
e^{-kt}\|u (0)\|_{-\theta_2}^2 + K,
\quad t\geq0,
\end{equation}
where the implicit constant may depend on $T'$.

b) In addition to the above hypotheses in a), for some $\beta\in[-\theta_2,\theta-\theta_2]$ let the following conditions be satisfied.
\begin{itemize}
\item[(iii)]
$b:V^{\beta}\times V^{\beta}\times V^{\theta-\beta}\ra\R$ is bounded;
\item[(iv)]
$\ang{Av,(I-\Delta)^{\beta} v}\geq c \|v\|_{\beta+\theta}^2$ for $v\in V^{\beta+\theta}$;
\item[(v)]
$\displaystyle\sup_{t\geq0}\|f\|_{L^2(t,t+T;V^{\beta-\theta})}^2\leq K$;
\item[(vi)]
$\dot{u}\in L^{2}_{\mathrm{loc}}(0,\infty;V^{\beta-\theta})$.
\end{itemize}
Then for any $t_0>0$
we have
\begin{equation}
\|u(t)\|_{\beta}^2
\lesssim
(e^{-k t} \|u (0)\|_{-\theta_2}^2 + K) \exp(e^{-k t} \|u (0)\|_{-\theta_2}^2 + K),
\quad t\geq t_0,
\end{equation}
where the implicit constant may depend on $t_0$.
\end{theorem}

\begin{remark}
The trilinear forms $b_1$ and $b_3$ satisfy condition (iii) provided
$2\theta > \frac{n+2}{2}-2\theta_1-\theta_2$,
$2\theta+2\theta_1 \geq 1-k$,
$2\theta_2+\theta\geq 1$ and
$2\theta_1+\theta \geq k$, for some $k\in\{0,1\}$.

The trilinear form $b_2$ satisfies condition (iii) provided
$2\theta > \frac{n+2}{2}-2\theta_1-\theta_2$,
$2\theta+2\theta_1 \geq 1$,
$2\theta_2+\theta\geq 0$ and
$2\theta_1+\theta \geq 1$.

The trilinear forms $b_4$ and $b_5$ satisfy condition (iii) provided
$2\theta > \frac{n+2}{2}-2\theta_1-\theta_2$,
$2\theta+2\theta_1 \geq 1$,
$2\theta_2+\theta\geq 1$ and
$2\theta_1+\theta \geq 1$.
\end{remark}

\begin{remark}
All the special cases listed in Table \ref{t:spec} (except NSE) satisfy condition (iii).
In the case of NS-$\alpha$-like model, condition (iii) is satisfied provided  $\theta \geq 1$ and $\ttwo > \frac{n+2}{2}-2$.
\end{remark}

\begin{proof}[Proof of Theorem \ref{t:attr-exist}]
We have
\begin{equation*}
\frac{d}{dt}\langle u, N u \rangle + 2 \langle Au, N u\rangle 
= 2 \langle f, N u \rangle 
\leq \eps^{-1} \|f\|_{-\theta-\theta_2}^2 + \eps \|N\|_{-\theta_2;\theta_2}^2 \|u\|_{\theta-\theta_2}^2,
\end{equation*}
for any $\eps>0$.
By using (i) and by choosing $\eps>0$ sufficiently small,
we get
\begin{equation}\label{e:attr-1}
\frac{d}{dt}\|u\|_{-\theta_2}^2 + c \|u\|_{\theta-\theta_2}^2 
\lesssim
\|f\|_{-\theta-\theta_2}^2.
\end{equation}
and since $V^{\theta-\theta_2}\hookrightarrow V^{-\theta_2}$ we have
\begin{equation}\label{e:attr-2}
\frac{d}{dt}\|u\|_{-\theta_2}^2 + k \|u\|_{-\theta_2}^2 
\lesssim
\|f\|_{-\theta-\theta_2}^2,
\end{equation}
with some constant $k>0$.
This gives
\begin{equation}\label{e:attr-3}
\begin{split}
\|u (t)\|_{-\theta_2}^2
& \lesssim
e^{-k t} \|u (0)\|_{-\theta_2}^2
 + 
\int_{0}^t 
e^{-k(t-\tau)} \|f(\tau)\|_{-\theta-\theta_2}^2 d\tau
\\
& \lesssim
e^{-k t} \|u (0)\|_{-\theta_2}^2
 + 
K,
\end{split}
\end{equation}
and by integrating \pref{e:attr-1} and using \pref{e:attr-3}, we have
\begin{equation}\label{e:attr-4}
\int_t^{t+T'} \|u\|_{\theta-\theta_2}^2
\lesssim
\|u(t)\|_{-\theta_2}^2 + 
\int_t^{t+T'} \|f\|_{-\theta-\theta_2}^2
\lesssim
e^{-k t} \|u (0)\|_{-\theta_2}^2
 + K,
\end{equation}
proving a).

Now we shall prove b).
As in the proof of Theorem~\ref{t:reg}, 
we get $u\in L^2(0,T;V^{\beta+\theta})$.
Taking $w=\Lambda^{2\beta}u$ in \pref{e:weak},
and using (iv) and the boundedness of $b$,
we have
\begin{equation}\label{e:attr-5}
\frac{d}{dt}\|u\|_{\beta}^2 
+ k' \|u\|_{\beta}^2
\lesssim
\|f\|_{\beta-\theta}^2 
+ \|u\|_{\theta-\theta_2}^2 \|u\|_{\beta}^2,
\end{equation}
with some constant $k'>0$, implying
\begin{equation}\label{e:attr-6}
\begin{split}
\|u (t)\|_{\beta}^2
\lesssim~
& e^{-k(t-s)} \exp(\|u\|_{L^2(s,t;V^{\theta-\theta_2})}^2) \|u (s)\|_{\beta}^2 \\
& + \int_{s}^t e^{-k(t-\tau)} \exp(\|u\|_{L^2(\tau,t;V^{\theta-\theta_2})}^2) \|f(\tau)\|_{\beta-\theta}^2  d\tau.
\end{split}
\end{equation}
Integrating this over $s\in [t-t_0,t]$ we have
\begin{equation*}
\begin{split}
t_0 \|u (t)\|_{\beta}^2
\lesssim~
& \exp(\|u\|_{L^2(t-t_0,t;V^{\theta-\theta_2})}^2) \int_{t-t_0}^{t} \|u (s)\|_{\beta}^2 ds \\
&+ t_0 \exp(\|u\|_{L^2(t-t_0,t;V^{\theta-\theta_2})}^2) \int_{t-t_0}^t \|f(\tau)\|_{\beta-\theta}^2  d\tau\\
\lesssim~
& (e^{-k t} \|u (0)\|_{-\theta_2}^2 + K + K) \exp(e^{-k t} \|u (0)\|_{-\theta_2}^2 + K),
\end{split}
\end{equation*}
where we have used \pref{e:attr-4} and (v).
This completes the proof.
\end{proof}
For example, in the case of ML-$\alpha$ model, conditions $i)-vi)$ of Theorem \ref{t:attr-exist} are satisfied with $\beta=0$.

In this next corollary, we combine the results in Theorem~\ref{t:exist}, Theorem~\ref{t:stab} and Theorem~\ref{t:attr-exist} to show the existence of a global attractor.
\begin{corollary}\label{c:attr-exist}
Let the following conditions hold.
\begin{itemize}
\item[i)] 
$b:V^{\sigma_1}\times V^{\theta-\theta_2}\times V^{\sigma_2}\ra\R$ is bounded for some 
$\sigma_1\leq\theta-\theta_2$ and $\sigma_2\leq\theta+\theta_2$ with $\sigma_1+\sigma_2\leq\theta$;
\item[ii)] 
$b(w,v,Nv)=0$ for any $v,w\in V^{\theta-\theta_2}$;
\item[iii)] 
$b:V^{\bar\sigma_1}\times V^{\bar\sigma_2}\times V^{\bar\gamma}\ra\R$ is bounded for some 
$\bar\sigma_i<\theta-\theta_2$, $i=1,2$, and $\bar\gamma\in\R$.
\end{itemize}
In addition, assume that the hypotheses (i) and (iii-v) of Theorem \ref{t:attr-exist} are satisfied.
Then, there exists a compact attractor $\mathcal{A}\Subset V^{-\theta_2}$ 
for the equation \pref{e:weak}
which attracts the bounded sets of $V^{-\theta_2}$.
Moreover, $\mathcal{A}$ is connected and it is the maximal bounded attractor in $V^{-\theta_2}$.
\end{corollary}

\begin{proof}
We recall that by Theorem \ref{t:exist}, there exists a solution $u\in L^{\infty}_{\mathrm{loc}}(0,\infty;V^{-\theta_2}) \cap L^{2}_{\mathrm{loc}}(0,\infty;V^{\theta-\theta_2})$ to \pref{e:weak} with any given initial data $u(0)\in V^{-\theta_2}$.
By Theorem \ref{t:stab} this solution is unique and depends continuously on the initial data,
so we have a continuous semigroup $S(t):V^{-\theta_2}\to V^{-\theta_2}$, $t\geq0$.
Now, by Theorem \ref{t:attr-exist} there is a ball $B$ in $V^{\theta-\theta_2}$
which is absorbing in $V^{-\theta_2}$, meaning that for any bounded set $U\subset V^{-\theta_2}$
there exists $t_1$ such that $S(t)U\subset B$ for all $t\geq t_1$.
Therefore for any bounded set $U\subset V^{-\theta_2}$
there exists $t_0$ such that $\cup_{t\geq t_0}S(t)U$ is relatively compact in $V^{-\theta_2}$.
Finally, applying~\cite[Theorem I.1.1.]{Tema88} we have that the set
$\mathcal{A}=\cap_{s\geq0}\overline{\cup_{t\geq s}S(t)B}$ is a compact attractor for $S$, and the rest of the result is immediate.
\end{proof}

All the special cases listed in Table \ref{t:spec} (except NSE) satisfy the conditions of Corollary~\ref{c:attr-exist}.
Again, in the case of NS-$\alpha$-like model, the conditions of the corollary are satisfied provided  $\theta \geq 1$ and $\ttwo > \frac{n+2}{2}-2$.

%%%%%%%%%%%%%%%%%%%%%%%%%%%%%%%%%%%%%%%%%%%%%%%%%%%%%%%
\subsection{Estimates on the dimension of the global attractor}
\label{ss:attr-dim}
%%%%%%%%%%%%%%%%%%%%%%%%%%%%%%%%%%%%%%%%%%%%%%%%%%%%%%%

Next we give a result which can be used to develop estimates on the 
dimension of the global attractor.
To obtain bounds on the dimension of global attractors, we require conditions that will guarantee that any $m$-dimensional volume element in the phase space shrinks as the flow evolves.
The general notion is that if this is the case, then the attractor can have no $m$-dimensional subsets and hence its dimension must be less than or equal to  $m$.
If one can find such an $m< \infty$, then we say that the asymptotic dynamics is determined by finite number of degrees of freedom.

\begin{theorem}
\label{thm:global-attractor}
Let $\theta>0$.
Let the equation \pref{e:weak} admit the semigroup $S(t):V^{-\theta_2}\to V^{-\theta_2}$, ${t\geq0}$,
and let $X\subset V^{\theta-\theta_2}$ be a bounded set such that $S(t)X=X$ for $t\geq0$.
Let the following conditions hold:
\begin{itemize}
\item[i)] 
$b:V^{\sigma_1}\times V^{\theta-\theta_2}\times V^{\sigma_2}\ra\R$ is bounded for some 
$\sigma_1\leq\theta-\theta_2$ and $\sigma_2\leq\theta+\theta_2$ with $\sigma_1+\sigma_2<2\theta$;
\item[ii)] 
$b(w,v,Nv)=0$ for any $v,w\in V^{\theta-\theta_2}$;
\item[iii)] 
For some $\beta\geq-\theta_2$ and $p\in[1,\infty]$,
\begin{equation}
\displaystyle\epsilon:=\sup_{u_0\in X}\limsup_{t\to\infty}\frac1t\int_{0}^{t}\|S(\tau)u_0\|_{\beta}^pd\tau<\infty;
\end{equation}
\item[iv)] 
For some $m\in\N$, $\alpha\in[0,1)$, $q\in[1,p]$, and $C>0$,
and for any collection $\{\phi_i\in V^{\theta-\theta_2}\}_{i=1}^{m}$ satisfying $\ang{\phi_i,N\phi_k}=\delta_{ik}$, $i,k=1,\ldots,m$, 
\begin{equation}
\sum_{i=1}^{m}b(\phi_i,S(t)u_0,N\phi_i)\leq\alpha\sum_{i=1}^{m}\ang{A\phi_i,N\phi_i} + C\|S(t)u_0\|_{\beta}^q,
\qquad
u_0\in X,\, t\geq0;
\end{equation}
\item[v)] 
For any collection $\{\phi_i\}$ as above,
\begin{equation}
(1-\alpha)\sum_{i=1}^{m}\ang{A\phi_i,N\phi_i}> C \epsilon^{\,q/p},
\end{equation}
\end{itemize}

Then we have $d_H(X)\leq
%d_F(\mathcal{A})\leq
m$.
\end{theorem}

\begin{proof}
Given $u_0\in X$,
the linearization of \pref{e:weak} around the solution $u(t)=S(t)u_0$, $t\geq0$, is
\begin{equation*}
\dot U + AU + B(u,U) + B(U,u) = 0,\qquad U(0)=U_0\in V^{-\theta_2}.
\end{equation*}
There exists a unique solution to the above equation,
and we will denote $U(t)=L(t,u_0)U_0$.
One can show that for any fixed $t\geq0$, $L(t,\cdot):V^{-\theta_2}\to V^{-\theta_2}$ is uniformly bounded on $X$, i.e.,
\begin{equation*}
\sup_{u_0\in X}\|L(t,u_0)\|_{-\theta_2;-\theta_2}<\infty.
\end{equation*}
Moreover one can prove that for any fixed $t\geq0$,
\begin{equation*}
\sup_{(\xi,\eta)\in D_\eps}\frac{\|S(t)\eta-S(t)\xi-L(t,\xi)(\eta-\xi)\|_{-\theta_2}}{\|\eta-\xi\|_{-\theta_2}}\to0
\quad\textrm{as }\eps\to0,
\end{equation*}
where $D_\eps=\{(\xi,\eta):\xi,\eta\in X,\,\|\xi-\eta\|_{-\theta_2}\leq\eps\}$.
Introducing the notation
\begin{equation*}
T(t,u_0)U= AU + B(u,U) + B(U,u),
\end{equation*}
where $u(t)=S(t)u_0$ is understood,
we have
\begin{equation*}
\begin{split}
\sum_{i=1}^{m}\ang{T(t,u_0)\phi_i,N\phi_i}
&=
\sum_{i=1}^{m}\ang{A\phi_i,N\phi_i}
+ \sum_{i=1}^{m}b(\phi_i,u,N\phi_i)
\\
&\geq
(1-\alpha) \sum_{i=1}^{m}\ang{A\phi_i,N\phi_i}
-C\|u\|_{\beta}^q,
\end{split}
\end{equation*}
implying that
\begin{equation*}
\begin{split}
\inf_{u_0\in X}\liminf_{t\to\infty}\frac1t\sum_{i=1}^{m}\ang{T(t,u_0)\phi_i,N\phi_i}
\geq~
&(1-\alpha) \sum_{i=1}^{m}\ang{A\phi_i,N\phi_i}
\\
&-C\sup_{u_0\in X}\left(\limsup_{t\to\infty}\frac1t\|u\|_{\beta}^p\right)^{\frac qp}
\\
>~ &
0.
\end{split}
\end{equation*}
Now we apply~\cite[Theorem V.3.3]{Tema88} (see also pp.\ 291 therein) to complete the proof.

\end{proof}

\begin{remark}
Theorem~\ref{thm:global-attractor} can be used to recover estimates on the dimension of the global attractor for the generalized model, through the application of techniques previously used in the literature for the special cases listed in Table~\ref{t:spec}; this is a somewhat long calculation that we do not include here.
\end{remark}

%%%%%%%%%%%%%%%%%%%%%%%%%%%%%%%%%%%%%%%%%%%%%%%%%%%%%%%
\section{Determining operators}
\label{s:determining}
%%%%%%%%%%%%%%%%%%%%%%%%%%%%%%%%%%%%%%%%%%%%%%%%%%%%%%%

The notion of {\em determining modes} for the Navier-Stokes and MHD equations 
was first introduced in~\cite{FoPr67} as an attempt to identify and estimate 
the number of degrees of freedom in turbulent flows (cf.~\cite{CFMT85} for a
thorough discussion of the role of determining sets in turbulence theory).
This concept later led to the notion of 
{\em Inertial Manifolds}~\cite{FST88}.
An estimate of the number of determining modes was 
given in~\cite{FMTT83,JoTi93};
the concepts of {\em determining nodes} and {\em determining volumes} 
were introduced and estimated in~\cite{FoTe83,FoTe84,FoTi91,JoTi91,JoTi93}.
See also~\cite{JoTi92,CJT95a,CJT97}.
In~\cite{HoTi96b,HoTi96a}, a more general concept known as a
{\em determining operator} was introduced, and the special
case of {\em determining functionals} was explicitly given.

Following~\cite{HoTi96b,HoTi96a,Hols95b},
we now define more precisely the concepts of
{\em determining operators} and {\em determining functionals}
for weak solutions of \pref{e:weak}.
In the following two definitions, we consider
an operator $R_m : V^{\theta-\theta_2} \to H_m$,
where $H_m\subset H^{\alpha}$ is a finite dimensional subspace with some $\alpha\leq-\theta_2$.

\begin{definition}
   \label{d:asympt-det-op}
Let $f, g \in L^2(0,\infty;V^{-\theta-\theta_2})$ be any two forcing functions satisfying
\begin{equation}
    \label{e:forcing}
\lim_{t \rightarrow \infty} \| f(t) - g(t) \|_{-\theta-\theta_2} = 0,
\end{equation}
and let $u, v \in L^2(0,\infty;V^{\theta-\theta_2})$ be corresponding solutions
to~\pref{e:weak}.
Then $R_m$ is called an {\em asymptotic determining operator} for weak solutions of \pref{e:weak} if
\begin{equation}
  \label{e:project}
\lim_{t \rightarrow \infty} 
    \| R_m [ u(t) - v(t) ] \|_{-\theta_2} = 0,
\end{equation}
implies that
\begin{equation}
    \label{e:determine}
\lim_{t \rightarrow \infty} \| u(t) - v(t) \|_{-\theta_2} = 0.
\end{equation}
\end{definition}

\begin{definition}
   \label{d:set-det-op}
With $K\subset V^{\theta-\theta_2}$,
let $u(t), v(t) \in K$, $t\in\R$, be solutions to~\pref{e:weak}.
Then $R_m$ is called a {\em determining operator on the set $K$} for weak solutions of \pref{e:weak} if
\begin{equation}
  \label{e:project-set}
    R_m u(t) = R_m v(t) , \qquad t\in\R,
\end{equation}
implies that $u=v$.
\end{definition}

Given a basis $\{ \phi_i \}_{i=1}^m$ for the finite-dimensional space $H_m$, 
and a set of bounded linear functionals $\{ l_i \}_{i=1}^m\subset V^{\theta_2-\theta}$, 
we can construct the operator
\begin{equation}
  \label{eqn:projectOper}
R_m u = \sum_{i=1}^m l_i(u) \phi_i.
\end{equation}
The assumption \pref{e:project} is then implied by:
\begin{equation}
  \label{e:linfunc}
\lim_{t \rightarrow \infty} 
    | l_i( u(t) - v(t)) | = 0, \ \ \ \ \ i=1, \ldots, m
\end{equation}
so that we can ask equivalently whether the set $\{l_i\}_{i=1}^m$ 
forms a set of {\em determining functionals}.
The analysis of whether $R_m$ or $\{l_i\}_{i=1}^m$ are determining 
can be reduced to an analysis of the approximation properties of $R_m$.
Note that in this construction,
the basis $\{\phi_i\}_{i=1}^m$ need not span a subspace of the solution 
space $V^{-\theta_2}$ or even of $H^{-\theta_2}$, so that the functions $\phi_i$ need not be divergence-free
or be in $H^{-\theta_2}$, for example.
Note that Definitions~\ref{d:asympt-det-op} and~\ref{d:set-det-op}
encompasses each of the notions of determining modes, nodes, volumes,
and functionals, by making particular choices for the sets of functions 
$\{\phi_i\}_{i=1}^m$ and $\{l_i\}_{i=1}^m$.

Here we extend the results of~\cite{HoTi96b,HoTi96a,Hols95b} 
to the generalized Navier-Stokes model \pref{e:op}.
In particular, we will show that if $\{H_m \subset H^{\alpha}:m\in\N\}$ is
a family of finite dimensional subspaces, and
if a family of operators 
$R_m : V^{\theta-\theta_2} \to H_m$, $m\in\N$,
satisfies an approximation inequality of the form
\begin{equation}
  \label{e:approximation}
\| u - R_m u \|_{\alpha} \le \xi(m) \| u \|_{\theta-\theta_2},
\end{equation}
for a function $\xi:(0,\infty)\to(0,\infty)$ with $\lim_{m\to\infty}\xi(m)=0$,
then the operator $R_m$ is a determining operator in the sense of 
Definitions~\ref{d:asympt-det-op} and~\ref{d:set-det-op}, provided $m$ is large enough.
%We will also derive explicit bounds on $m$ which guarantees that $R_m$ is determining.

If
$H_m$ both contains 
all polynomials of degree less than $\lceil\theta-\theta_2\rceil$, and is spanned by compactly supported 
functions such that the diameter of the supports is uniformly proportional to $m^{-1/n}$,
we typically have $\xi(m)\sim m^{-({\theta-\ttwo-\alpha})/{n}}$,
provided that $R_m$ realizes a near-best approximation of any $u\in V^{\theta-\theta_2}$ from the subspaces $H_m$ in the $H^\alpha$-norm.
In particular, standard finite element and wavelet subspaces of sufficiently high polynomial degree satisfy these conditions.
Then $R_m$ may be chosen to be interpolation or quasi-interpolation operators, cf.\cite{Hols95b,HoTi96b}.
For example, the piecewise constants with local averaging and piecewise linears
with e.g.\ the Scott-Zhang quasi-interpolators as in~\cite{Hols95b,HoTi96b} correspond to the determining volumes and determining nodes, respectively.

Determining modes can be understood as follows.
Let $\Lambda:V^{\theta-\theta_2}\to V^{\alpha}$ be a boundedly invertible operator
with the eigenvalues $0<\lambda_1\leq\lambda_2\leq\ldots$,
such that $\lambda_j\to\infty$ as $j\to\infty$.
Let $\{w_j\}\subset V^{\theta-\theta_2}$ be the corresponding set of eigenfunctions, orthonormal in $V^\alpha$.
Then, with $H_m:=\mathrm{span}\{w_1,\ldots,w_m\}$
and $R_m$ the $V^\alpha$-orthogonal projector onto $H_m$,
it is easy to see that $\|u-R_mu\|_\alpha\lesssim \lambda_m^{-1}\|u\|_{\theta-\theta_2}$ for any $u\in V^{\theta-\theta_2}$,
meaning that $\xi(m)\sim\lambda_m^{-1}$ in this case.
In particular, when $\Lambda$ is a power of the Stokes operator, i.e,
when $\Lambda=(-P\Delta)^{(\theta-\theta_2-\alpha)/2}$, we have $\lambda_m\sim m^{(\theta-\theta_2-\alpha)/n}$, and so $\xi(m)\sim m^{-({\theta-\ttwo-\alpha})/{n}}$, which coincides with the behavior of $\xi$ for the case in the previous paragraph.

Bounds on the number of determining degrees of freedom are usually 
phrased in terms of a generalized {\em Grashof number}, which can
be defined in the current context as
\begin{equation}\label{e:grash}
G = \limsup_{t \rightarrow \infty} \| f(t) \|_{V^{-\theta-\theta_2}}.
\end{equation}
The definition in \pref{e:grash} generalizes the definition of Grashof number in the literature.
For example, if the forcing term in NSE is given by $\tilde{f}$ with dimensions  mass $\times$ length $\times$ time$^{-2}$, then the nondimensional forcing $f$ that appears in \pref{e:op} can be defined in terms of $\tilde{f}$ by
\begin{equation}\label{e:f}
f = \frac{L^2}{\rho\nu^2}\tilde{f},
\end{equation}  
where $\rho$ is the density of dimensions mass per unit $n$-volume, $\nu$ is the kinematic viscosity with dimensions length$^2\times$ time$^{-1}$ and $L$ is the system size.
In this case one can see that given a time independent forcing for NSE, the Grashof number is defined as 
\begin{equation}\label{e:Gnse}
G=\frac{L^{2-n/2}}{\rho\nu^2}\|\tilde{f}\|_{-1}.
\end{equation}
It is known that if $G$ is small enough, then the NSE possess a unique, globally stable, steady state solution~\cite{Tema77}.
As the Grashof number increases, the steady state goes through a sequence of bifurcations leading to a more complex dynamics of the flow.
Hence, it is natural to use the Grashof number $G$ to estimate the number of degrees of freedom of the solutions of the NSE as well as other turbulence models.

%%%%%%%%%%%%%%%%%%%%%%%%%%%%%%%%%%%%%%%%%%%%%%%%%%%%%%%
\subsection{Dissipative systems}
\label{ss:dis}
%%%%%%%%%%%%%%%%%%%%%%%%%%%%%%%%%%%%%%%%%%%%%%%%%%%%%%%

In this subsection, we consider equations with $\theta>0$.
Note that Theorem \ref{t:attr-exist} provides with examples where the conditions iii) of the following theorem is satisfied (with $\beta=\theta-\theta_2$ and $p=2$).

\begin{theorem}
   \label{t:dis-det}
(a) Let $\theta>0$, and let $u,v\in L^{\infty}(0,\infty;V^{-\theta_2}) \cap L^2(0,\infty;V^{\theta-\theta_2})$ be two weak solutions of \pref{e:op} with the forcing functions 
$f, g \in L^2(0,\infty;V^{-\theta-\theta_2})$, respectively.
Let $R_m : V^{\theta-\theta_2} \to H_m \subset H^{-\theta_2}$, $m\in\N$, be a family of operators
satisfying the approximation property \pref{e:approximation} with $\alpha=-\theta_2$.
In addition, with $\beta\geq\theta-\theta_2$, let the following conditions be fulfilled.
\begin{itemize}
\item[i)]
$b:V^{\sigma_1}\times V^{\beta}\times V^{\sigma_2}\ra\R$ is bounded for some 
$\sigma_1\leq\theta-\theta_2$ and $\sigma_2\leq\theta+\theta_2$ with $\sigma_1+\sigma_2<\theta$;
\item[ii)]
$b(w,z,Nz)=0$ for all $w\in V^{\sigma_1}$ and $z\in V^{\beta}$;
\item[iii)]
$\displaystyle\epsilon:=\inf_{T>0}\limsup_{t\to\infty}\frac1T\int_{t}^{t+T}\|u(\tau)\|_{\beta}^pd\tau<\infty$ with $p=\frac{\theta}{\theta-\sigma_1-\sigma_2}$;
\item[iv)]
$\displaystyle\lim_{t \rightarrow \infty} \| f(t) - g(t) \|_{-\theta-\theta_2} = 0$;
\item[v)] 
$\displaystyle\lim_{t \rightarrow \infty} 
    \| R_m [ u(t) - v(t) ] \|_{-\theta_2} = 0$ for some $m$ satisfying 
\\
$\xi(m)^{-2} > \dfrac{4C_A}{c_A} + \dfrac{4(p-1)\|b\|^2}{p^2c_A^2} \epsilon$.
\end{itemize}
Then we have
$$
\lim_{t \rightarrow \infty} \| u(t) - v(t) \|_{-\theta_2} = 0.
$$

(b) Assume all of the above hypotheses with the time interval $[0,\infty)$ replaced by $\R$, and the conditions iii)-v) are replaced by
that $f=g$, $\epsilon:=\sup_{t\in\R}\|u(\tau)\|_{\beta}^p<\infty$ with $p$ as above,
and that $R_m u(t) = R_m v(t)$ for all $t\in\R$ with $m$ as above.
Then we have $u=v$.
\end{theorem}

\begin{proof}
Let $w=u-v$.
Subtracting the equations~\pref{e:op} for $u$ and $v$ yields
\begin{equation}
  \label{eqn:above}
\frac{dw}{dt} + A w + B(u,u) - B(v,v) = f-g.
\end{equation}
Pairing this with $Nw$, and by using condition ii), we get
\begin{equation}\label{eq:abv}
\dfrac{1}{2}\frac{d}{dt}\ang{w,Nw} + \ang{Aw, Nw} = \ang{f-g,Nw} - b(w,u,Nw) 
\end{equation}
Using Young's inequality, one can estimate the right hand side of equation \pref{eq:abv} as follows:
\begin{equation}\label{e:abv1}
\begin{split}
\dfrac{d}{dt}\|w\|_N^2 + 2\ang{Aw,Nw} \leq & \dfrac{1}{\delta}\|f-g\|^2_{-\theta-\theta_2} + \delta \|N\|^2_{-\theta_2;\theta_2}\|w\|_{\theta-\theta_2}^2 \\
& + \|b\|\|w\|_{\sigma_1}\|u\|_{\beta}\|w\|_{\sigma_2-2\theta_2},
\end{split}
\end{equation}
where we denote $\|w\|_{N}^2:=\ang{w,Nw}$.
We estimate the last term as follows:
\begin{equation}
\|w\|_{\sigma_1} \|u\|_\beta\|w\|_{\sigma_2-2\theta_2}\leq \|w\|^{2-\lambda_1-\lambda_2}_{-\theta_2} \|w\|_{\theta-\theta_2}^{\lambda_1+\lambda_2}\|u\|_{\beta},
\end{equation}
where $\lambda_1 = \frac{\sigma_1 + \theta_2}{\theta}$ and $\lambda_2 = \frac{\sigma_2-\theta_2}{\theta}$.
Using Young's inequality, we get
\begin{equation*}\label{eqn:abv2}
\|w\|_{\sigma} \|u\|_\beta\|w\|_{\sigma_2-2\theta_2}\leq \dfrac{\veps}{q}\|w\|^{(\lambda_1 + \lambda_2)q}_{\theta-\theta_2} + \dfrac{1}{p\veps}\|w\|^{(2-\lambda_1-\lambda_2)p}_{-\theta_2} \|u\|_\beta^p.
\end{equation*}
Note that $(\lambda_1 + \lambda_2)q=(2-\lambda_1-\lambda_2)p=2$.
Let us now choose $\delta=\frac{c_A}{2\|N\|_{-\theta_2;\theta_2}^2}$ and $\veps=\frac{qc_A}{2\|b\|}$, then it follows, taking into account the coercivity of $A$ that 
\begin{equation*}
\dfrac{d}{dt}\|w\|_N^2 + c_A\|w\|^2_{-\theta-\theta_2} - 2C_A\|w\|_{-\theta_2}^2
\leq 
\dfrac{1}{\delta}\|f-g\|^2_{-\theta-\theta_2} +  \dfrac{\|b\|}{p\veps}\|w\|^{2}_{-\theta_2} \|u\|_\beta^p.
\end{equation*}
To bound the second term on the left from below, we employ the approximation assumption on $R_m$, which yields
\begin{equation}\label{e:abv3}
\begin{split}
\dfrac{d}{dt}\|w\|_N^2 & + \left( \frac{c_A}{2}\xi(m)^{-2} - 2C_A - \dfrac{\|b\|}{p\veps} \|u\|_\beta^p \right) \|w\|^2_{-\theta_2} 
\\
& \leq \dfrac{1}{\delta}\|f-g\|^2_{-\theta-\theta_2}
+ \xi(m)^{-2} \|R_m w\|^2_{-\theta_2}.
\end{split}
\end{equation}
This is of the form
$$
\frac{d}{dt} \|w\|_{-\theta_2}^2 + x \|w\|_{-\theta_2}^2 \le y,
$$
with obvious definition of $x$ and $y$.

Lemma~\ref{lemma:gronwall_2} can now be applied.
Recall both $\|f-g\|_{-\theta-\theta_2} \rightarrow 0$ and
$\|R_m w\|_{-\theta_2} \rightarrow 0$ 
as $t \rightarrow \infty$ by assumptions iv) and v).
So taking into account iii) we have
$$
\lim_{t \rightarrow \infty}
     \int_{t}^{t+T} y^+(\tau) d\tau = 0,
\ \ \ \ \ \ \ \ 
\limsup_{t \rightarrow \infty}
    \int_{t}^{t+T} x^-(\tau) d\tau <\infty.
$$
It remains to verify that for some $T>0$,
$$
\liminf_{t \rightarrow \infty}
    \int_{t}^{t+T} x(\tau) d\tau >0.
$$
This means we must verify the following inequality for some $T > 0$:
\begin{equation}
   \label{eqn:key}
\xi(m)^{-2}
> 
\frac{4C_A}{c_A} + \dfrac{2\|b\|}{p\veps c_A} 
\limsup_{t \rightarrow \infty}
\frac{1}{T} \int_{t}^{t+T} \|u\|_\beta^p d\tau.
\end{equation}
Therefore, if
\begin{equation}
\xi(m)^{-2}
> 
\frac{4C_A}{c_A} + \dfrac{2\|b\|}{p\veps c_A} \epsilon,
\end{equation}
implying that~\pref{eqn:key} holds for some $T>0$, then by Lemma~\ref{lemma:gronwall_2}, it follows that
$$
\lim_{t \rightarrow \infty} \| u(t) - v(t) \|_{-\theta_2} 
= \lim_{t \rightarrow \infty} \|w(t)\|_{-\theta_2}
= 0.
$$
This completes the proof of (a).

For (b), the right hand of \pref{e:abv3} vanishes, implying
$$
\frac{d}{dt} \|w\|_{-\theta_2}^2 + k \|w\|_{-\theta_2}^2 \le 0,
$$
with some $k>0$.
The Gr\"onwall inequality gives
$$
\|w(t)\|_{-\theta_2}^2 \leq e^{k(s-t)} \|w(s)\|_{-\theta_2}^2,
$$
for any $t\geq s$, and now sending $s\to-\infty$ we get the conclusion $w(t)=0$.
\end{proof}

\begin{remark}
The trilinear forms $b_1$ and $b_3$ satisfy the hypothesis of a) provided 
$\frac{3\theta}{2}+ 2\tone+\ttwo > \frac{n+2}{2}$, 
$\theta + \theta_2 \geq \frac{1}{2}$,
$\frac{\theta}{2}+2\theta_1 > k$, and
$\theta+\theta_1 \geq \frac{1-k}{2}$, for some $k\in\{0,1\}$.

The trilinear form $b_2$ satisfies the hypothesis of a) provided 
$\frac{3\theta}{2}+ 2\tone+2\ttwo > \frac{n+2}{2}$, 
$\theta + \theta_2 \geq 0$,
$\frac{\theta}{2}+2\theta_1 > 1$, and
$\theta+\theta_1 \geq \frac{1}{2}$.

The trilinear forms $b_4$ and $b_5$ satisfy the hypothesis of a) provided 
$\frac{3\theta}{2}+ 2\tone+\ttwo > \frac{n+2}{2}$, 
$\theta + \theta_2 \geq \frac{1}{2}$,
$\frac{\theta}{2}+2\theta_1 > 1$, and
$\theta+\theta_1 \geq \frac{1}{2}$
\end{remark}

\begin{remark}
From \pref{energy_f}, we have $\displaystyle\epsilon\simeq \|f\|_{V^{-\theta-\ttwo}}^2\simeq G^2$ with $\beta=0$ and $p=2$.
Then condition $v)$ of Theorem \ref{t:dis-det} is equivalent to the condition $\displaystyle\xi(m)^{-2}\geq cG^2$.
Assuming that $\displaystyle\xi(m)\simeq m^{-({\theta-\ttwo-\alpha})/{n}}$,  and putting $\alpha=-\ttwo$, we have  $\displaystyle\xi(m)\simeq m^{-{\theta}/{n}}$.
Hence $\displaystyle m\gtrsim G^{{n}/{\theta}}$.
\end{remark}

%%%%%%%%%%%%%%%%%%%%%%%%%%%%%%%%%%%%%%%%%%%%%%%%%%%%%%%
\subsection{Nondissipative systems}
\label{ss:non-dis1}
%%%%%%%%%%%%%%%%%%%%%%%%%%%%%%%%%%%%%%%%%%%%%%%%%%%%%%%

In this subsection, we consider non-dissipative systems, which are represented in our generalized model when $\theta=0$.

\begin{theorem}
   \label{t:nondis-det}
(a) Let $u,v\in L^{\infty}(0,\infty;V^{-\theta_2})$ be two solutions of \pref{e:weak} with the forcing functions
$f, g \in L^2(0,\infty;V^{-\theta_2})$, respectively,
and with $\theta=0$ and
$$
\ang{Av,Nv}\geq c_A \|v\|_{-\theta_2}^2,
\qquad v\in V^{-\theta_2}.
$$
For some $\alpha\leq-\theta_2$, let $R_m : V^{-\theta_2} \to H_m \subset H^{\alpha}$, $m\in\N$, be a family of operators
satisfying the approximation property \pref{e:approximation}.
In addition, with $\beta\geq-\theta_2$, let the following conditions be fulfilled.
\begin{itemize}
\item[i)]
$b:V^{\alpha}\times V^{\beta}\times V^{\theta_2}\ra\R$ is bounded;
\item[ii)]
$b(w,z,Nz)=0$ for all $w\in V^{\alpha}$ and $z\in V^{\beta}$;
\item[iii)]
$\displaystyle\epsilon:=\limsup_{t\to\infty}\|u(t)\|_{\beta}<\infty$;
\item[iv)]
$\displaystyle\lim_{t \rightarrow \infty} \| f(t) - g(t) \|_{-\theta_2} = 0$;
\item[v)] 
$\displaystyle\lim_{t \rightarrow \infty} 
    \| R_m [ u(t) - v(t) ] \|_{\alpha} = 0$ for some $m$ satisfying $\displaystyle\xi(m) < \frac{c_A}{\|b\|\epsilon}$.
\end{itemize}
Then we have
$$
\lim_{t \rightarrow \infty} \| u(t) - v(t) \|_{-\theta_2} = 0.
$$

(b) Assume all of the above hypotheses with the time interval $[0,\infty)$ replaced by $\R$, and the conditions iii)-v) are replaced by
that $f=g$,  $\epsilon:=\sup_{t\in\R}\|u(\tau)\|_{\beta}<\infty$,
and that $R_m u(t) = R_m v(t)$ for all $t\in\R$ with $m$ as above.
Then we have $u=v$.
\end{theorem}

\begin{proof}
We start as in the proof of Theorem \ref{t:dis-det}, but instead of \pref{e:abv1} we get the following.
\begin{equation*}
\dfrac{d}{dt}\|w\|_N^2 + 2\ang{Aw,Nw} \leq \dfrac{1}{\delta}\|f-g\|^2_{-\theta_2} + \delta \|N\|^2_{-\theta_2;\theta_2}\|w\|_{-\theta_2}^2 + \|b\|\|w\|_{\alpha}\|u\|_{\beta}\|w\|_{-\theta_2}.
\end{equation*}
Let us now choose $\delta=\frac{c_A}{\|N\|_{-\theta_2;\theta_2}^2}$, then it follows, taking into account the coercivity of $A$ that 
\begin{equation*}
\dfrac{d}{dt}\|w\|_N^2 + c_A\|w\|^2_{-\theta_2}
\leq 
\dfrac{1}{\delta}\|f-g\|^2_{-\theta_2} + \|b\|\|w\|_{\alpha}\|u\|_{\beta}\|w\|_{-\theta_2}.
\end{equation*}
To bound the last term from above, we employ the approximation assumption on $R_m$, which yields
\begin{equation}
\begin{split}
\dfrac{d}{dt}\|w\|_N^2 & + \left( c_A - \xi(m) \|b\| \|u\|_\beta \right) \|w\|^2_{-\theta_2} \\
& \leq 
\dfrac{1}{\delta}\|f-g\|^2_{-\theta_2} + \|R_m w\|_{\alpha} \|b\| \|u\|_{\beta}\|w\|_{-\theta_2}.
\end{split}
\end{equation}
This is of the form
$$
\frac{d}{dt} \|w\|_{-\theta_2}^2 + x \|w\|_{-\theta_2}^2 \le y + y \|w\|_{-\theta_2},
$$
and an application of Lemma~\ref{lemma:gronwall_2} completes the proof of (a).
Part (b) is proven following the same argument as in the proof of Theorem \ref{t:dis-det}(b).
\end{proof}

\begin{remark}
Note that in the case of the NSV we need $-\frac32\leq\alpha\leq -1 $ in order to satisfy Theorem~\ref{t:nondis-det} $i)$ and the condition $\alpha\leq -1$.
Choosing $\alpha=-\frac32$ we get $m\gtrsim G^{6}$.
This is consistent with the calculations in~\cite{KT07}.
\end{remark}

\begin{remark}
The trilinear forms $b_1$ and $b_3$ satisfy the hypothesis i) of Theorem~\ref{t:nondis-det} provided 
$\beta+ 2\tone+2\ttwo > \frac{n+2}{2}$, 
$\beta + 3\theta_2 \geq 1$,
$2\theta_1 > k$, and
$2\theta_1+\beta+\ttwo \geq 1-k$, for some $k\in\{0,1\}$.
Also, we need that 
$\alpha> \frac{n+2}{2}-2\theta_1 -\beta-3\ttwo$,
$\alpha\geq k-2\tone-\ttwo$, and 
$\alpha\geq 1-k-2\tone-\beta-2\ttwo$, for some $k\in\{0,1\}$.
\end{remark}

\begin{remark}
The trilinear form $b_2$ satisfies the hypothesis i) of Theorem~\ref{t:nondis-det} provided 
$\beta+ 2\tone+2\ttwo > \frac{n+2}{2}$, 
$\beta + 3\theta_2 \geq 0$,
$2\theta_1 > 1$, and
$2\theta_1+\beta+\ttwo \geq 1$.
Also, we need that 
$\alpha> \frac{n+2}{2}-2\theta_1 -\beta-3\ttwo$,
$\alpha\geq 1-2\tone-\ttwo$, and 
$\alpha\geq 1-2\tone-\beta-2\ttwo$.
\end{remark}

\begin{remark}
The trilinear forms $b_4$ and $b_5$ satisfy the hypothesis i) of Theorem~\ref{t:nondis-det} provided 
$\beta+ 2\tone+2\ttwo > \frac{n+2}{2}$, 
$\beta + 3\theta_2 \geq 1$,
$2\theta_1 > 1$, and
$2\theta_1+\beta+\ttwo \geq 1$.
Also, we need that 
$\alpha> \frac{n+2}{2}-2\theta_1 -\beta-3\ttwo$,
$\alpha\geq 1-2\tone-\ttwo$, and 
$\alpha\geq 1-2\tone-\beta-2\ttwo$.
\end{remark}

\begin{remark}
  From \pref{energy_f} $\epsilon\simeq \|f\|_{V^{-\ttwo}}\simeq G$.
Then condition v) of Theorem \ref{t:dis-det} is equivalent to the condition $\xi(m)^{-1}\leq  cG$.
Assuming that $\xi(m)\simeq m^{-({\theta-\ttwo-\alpha})/{n}}$,  and putting $\theta=0$, we have $\xi(m)\simeq m^{({\ttwo+\alpha})/{n}}$.
Hence $m\gtrsim G^{-{n}/({\theta_2+\alpha})}$, where $\alpha\leq-\ttwo$.
\end{remark}

%%%%%%%%%%%%%%%%%%%%%%%%%%%%%%%%%%%%%%%%%%%%%%%%%
\section{Length-scale estimates in terms of the Reynolds number}
\label{s:Reynolds}

%%%\mnote{ All of Section \ref{s:Reynolds} is added in response to referee 1.  For brevity we did not put the whole section in bold face font.}

In the previous section, we established estimates on the number of degrees of freedom in terms of the generalized Grashof number $G$, a dimensionless parameter which measures the relative magnitude of forcing to the viscosity $\nu$.
A complementary scheme was introduced by Doering and Foias in~\cite{DF02}, to recast all the estimates in terms of the Reynolds number $Re$.
The Reynolds number measures the effect of nonlinearity in the fluid response, and in the current setting it allows us to measure the effects of modifying the nonlinearity.
It is important to recognize that in the engineering and physics communities, the Reynolds number is used more frequently than the Grashof number, as is viewed as more directly physically meaningful.
Therefore, in this section we will derive a lower bound on the Kolmogorov dissipation length-scale in terms of the Reynolds number, which will help provide some tools for relating the Grashof number-based results of the previous section to analogous statements involving the Reynolds number.

To set some notation, we briefly review the ideas discussed in~\cite{DF02,GH06} and then apply the analogous procedure to our more general problem.
Given the velocity field $\uns$ for the Navier-Stokes equations taken on an $n$-dimensional periodic domain $[0,L]^n$ with divergence free condition, the Reynolds number $Re$ of the flow is defined as
\begin{equation}\label{e:re}
Re = \frac{U l}{\nu}, \quad\mbox{where}\quad U^2 = L^{-n}\overline{\|\uns\|_0^2,} 
\end{equation}
where the overline denotes the long-time average
\begin{equation}
\overline{g} = \limsup_{t\rightarrow\infty} \frac{1}{t}\int_0^t g(\tau) \; d\tau,
\end{equation}
and $l$ is characterized by the following ``narrow-band type'' assumption on $f$
\begin{equation}
\|\nabla^{r}f\|_0 \simeq l^{-r}\|f\|_0.
\end{equation}
Recall the standard definition of Grashof number in $n$ dimensions in terms of the ``root mean square'' of the force
\begin{equation}\label{e:Gr_standard}
G = \frac{l^3f_{\textrm{rms}}}{\nu^2},
\quad\mbox{where}\quad 
f_{\textrm{rms}} = L^{-n/2}\|f\|_0
\end{equation}
Doering and Foias showed recently in~\cite{DF02} that in the limit $G\rightarrow\infty$, the solutions of the $n$-dimensional Navier-Stokes equations satisfy 
\begin{equation}\label{e:gr_estimate}
G \lesssim Re^2 + Re.
\end{equation}
The above estimate gives a way to transform any estimate given in terms of $G$ into an analogous estimate given in terms of $Re$.
However, as the following example from~\cite{GH06,GH08} shows, this procedure does not always give sharp estimates.
Consider the problem of bounding the Kolmogorov dissipation length-scale for the Navier-Stokes equations from below.
The Kolmogorov dissipation length-scale in terms of the energy dissipation rate $\eps$ is given by
\begin{equation}
l_d= (\nu^3/\eps)^{1/4},
\quad\textrm{where}\quad
\eps = \nu L^{-n}\overline{\|\nabla\uns\|_0^2} \lesssim G^2.
\end{equation}
Then, using \eqref{e:gr_estimate} we have the following bound
\begin{equation}\label{e:kolmogorov}
l\cdot l_d^{-1} \lesssim Re.
\end{equation}
In the three-dimensional case, one can obtain an estimate for the number of degrees of freedom in turbulent flows by dividing a typical length-scale of the flow by $l_d$ and taking the third power.
Thus, \eqref{e:kolmogorov} gives an upper bound of $Re^3$ to the number of degrees of freedom in turbulent flows which is not sharp compared to the generally accepted $Re^{9/4}$.
The authors of~\cite{DF02,GH06} obtained the bound of $Re^{9/4}$ by time-averaging the Leray's energy inequality and using \eqref{e:gr_estimate}.

In comparing estimates for the Navier-Stokes equations given in terms of the Reynolds number (and other quantities) to similar estimates for regularized equations, there is a basic issue of identifying precisely what the Reynolds number is for the regularized equations.
Here we will extend the approach followed in~\cite{GH06,GH08}, and identify the Navier-Stokes velocity field as 
$\uns = Nu$ (or $\uns=Mu$ if $M$ is more smoothing than $N$), where $u$ is the regularized velocity field, and define the Reynolds number and the energy dissipation rate $\eps$ in terms of $\uns$.
We have then
\begin{equation*}
\|\uns\|_0=\|Nu\|_0\eqsim\|u\|_{-2\theta_2},
\quad\textrm{or}\quad
\|\uns\|_0=\|Mu\|_0\eqsim\|u\|_{-2\theta_1},
\end{equation*}
and similarly, $\|\nabla\uns\|_0\eqsim\|u\|_{1-2\theta_2}$ or $\|\nabla\uns\|_0\eqsim\|u\|_{1-2\theta_1}$.
In other words, the approach of~\cite{GH06,GH08} naturally extends to our more general setting here, giving definitions of $Re$ and $\eps$ that satisfy
\begin{equation*}
Re^2\eqsim U^2\eqsim \overline{\|u\|_{-2\check\theta}^2},
\qquad\mbox{and}\qquad
\eps\eqsim\overline{\|u\|_{1-2\check\theta}^2},
\qquad\textrm{with}\quad
\check\theta=\max\{\theta_1,\theta_2\}.
\end{equation*}
The constants in the definitions are irrelevant as far as the asymptotic behavior of the bounds are concerned.
In the following theorem we derive a bound on a mean square Sobolev norm of $u$ in terms of a similar norm with a lower Sobolev order.
This will make possible the corollary following the theorem, which gives length-scale estimates in terms of the Reynolds number for the general setting in this paper.
This makes it possible to relate the bounds of~\S\ref{s:determining} giving in terms of the Grashof number to analogous bounds given in terms of the Reynolds number.

\begin{theorem}\label{t:length}
Let $u\in L^{\infty}(0,\infty;V^{-\theta_2})\cap L^{2}_{\textrm{\em loc}}(0,\infty;V^{\theta-\theta_2})$
be a solution to \eqref{e:op},
and let $\alpha<\beta\leq\theta-\theta_2$.
Let the following conditions hold:
\begin{itemize}
\item[i)] 
$b:V^{\alpha}\times V^{\alpha}\times V^{\gamma}\ra\R$ is bounded for some 
$\gamma$;
\item[ii)] 
$b(v,v,Nv)=0$ for $v\in V^{\theta-\theta_2}$;
\item[iii)] 
$\ang{Av,Nv}\geq c \|v\|_{\theta-\theta_2}^2$ for any $v\in V^{\theta-\theta_2}$, with a constant $c>0$;
\item[iv)] 
The forcing term $f$ is independent of $t$, and satisfies
\begin{equation*}
\|f\|_s
\leq
C(s)
\|f\|_0,
\end{equation*}
for any $s$, with constants $C(s)$ depending on $s$.
\end{itemize}
Then we have
\begin{equation*}
\overline{\|u\|_{\beta}^2}
\lesssim
\overline{\|u\|_{\alpha}^2}
+
\left(\overline{\|u\|_{\alpha}^2}\right)^{1+\frac{\beta-\alpha}{2(\theta-\theta_2-\alpha)}}
\end{equation*}
\end{theorem}

\begin{proof}
Pairing \eqref{e:op} with $Nu$, and by using ii), we get
\begin{equation*}
\ang{Au, Nu} = \ang{f,Nu} - \frac{1}{2}\frac{d}{dt}\ang{u,Nu}.
\end{equation*}
Then integrating in time gives
\begin{equation}\label{e:bnd-uf}
\int_0^t\|u\|_{\theta-\theta_2}^2 \lesssim \int_0^t\|u\|_{\alpha}\|f\|_{-2\theta_2-\alpha}  + \|u\|_{L^\infty(0,t;V^{-\theta_2})}^2.
\end{equation}
In terms of long time averages, from the above one can derive
\begin{equation*}
\overline{\|u\|_{\theta-\theta_2}^2} 
\lesssim 
\|f\|_{-2\theta_2-\alpha} \sqrt{\overline{\|u\|_{\alpha}^2}}
\lesssim
\|f\|_0U,
\end{equation*}
where
$U^2=\overline{\|u\|_{\alpha}^2}$, and we have used iv) in the last step.

Now we will bound $\|f\|_0$ in terms of $u$.
To this end, let us pair \eqref{e:op} with $f$, and write
\begin{equation*}
\ang{f,f} = \ang{Au, f} + b(u,u,f) + \frac{d}{dt}\ang{u,f}.
\end{equation*}
Recalling that $f$ is independent of $t$, and integrating in $t$, we have
\begin{equation}\label{e:bnd-f}
t\|f\|_0^2 \lesssim \int_0^t\|u\|_\alpha\|f\|_{2\theta-\alpha} + \int_0^t\|u\|_\alpha^2\|f\|_{\gamma} + \|u\|_{L^\infty(0,t;V^{-\theta_2})}\|f\|_{\theta_2},
\end{equation}
which implies
\begin{equation*}
\|f\|_{0} \lesssim  U+U^2.
\end{equation*}
Now plugging \eqref{e:bnd-f} into \eqref{e:bnd-uf}, we have
\begin{equation}\label{e:bnd-u}
\int_0^t\|u\|_{\theta-\theta_2}^2 
\lesssim 
\int_0^t\|u\|_{\alpha}\cdot
\frac1t
\left(
\int_0^t\|u\|_\alpha + \int_0^t\|u\|_\alpha^2 + O(1)
\right)
+ 
O(1),
\end{equation}
giving
\begin{equation*}
\overline{\|u\|_{\theta-\theta_2}^2} 
\lesssim
U^2+U^3.
\end{equation*}
Finally, we write the interpolation inequality
\begin{equation*}
\|u\|_{\beta} \leq \|u\|_{\alpha}^{1-\lambda}\|u\|_{\theta-\theta_2}^{\lambda},
\qquad\textrm{with }
\lambda=\frac{\beta-\alpha}{\theta-\theta_2-\alpha},
\end{equation*}
and calculate its long time average to establish the proof.
\end{proof}

Now we apply the above result to the situation where $\alpha=-2\theta_1$ and $\beta=1-2\theta_1$, and the trilinear form $b$ is given by $b_1$ with $\theta_1=\theta_2$.

\begin{corollary}\label{c:length}
Let the conditions ii), iii) and iv) of the previous theorem be satisfied,
and let $b=b_1$ with $\theta_1=\theta_2$.
Then we have
\begin{equation*}
\varepsilon\lesssim U^2 + U^{2+\frac{1}{\theta+\theta_1}},
\qquad\textrm{where}\quad
\varepsilon=\overline{\|u\|_{1-2\theta_1}^2},
\quad\textrm{and}\quad
U^2=\overline{\|u\|_{-2\theta_1}^2}.
\end{equation*}
In terms of the Kolmogorov length-scale and the Reynolds number, this is
\begin{equation*}
l_d^{-1}\lesssim Re^{\frac{1}{4}\left(2+\frac{1}{\theta+\theta_1}\right)} + Re^{\frac12}.
\end{equation*}
\end{corollary}

\begin{remark}
a) For the simplified Bardina model, our result is consistent with that of~\cite{GH08}, where the authors derive
\begin{equation*}
l_d^{-1} \lesssim Re^{5/8}.
\end{equation*}

b) For the hyper-viscous Navier-Stokes equation and the Navier-Stokes-Voight model, we have the estimates 
\begin{equation*}
l_d^{-1} \lesssim Re^{9/4},
\qquad\textrm{and}\qquad
l_d^{-1} \lesssim Re^{3/4},
\end{equation*}
respectively, that appear to be new.
\end{remark}

%%%%%%%%%%%%%%%%%%%%%%%%%%%%%%%%%%%%%%%%%%%%%%%%%%%%%%%

%% file: conc.tex
%%%%%%%%%%%%%%%%%%%%%%%%%%%%%%%%%%%%%%%%%%%%%%%%%%%%%%%
\section{Summary}
\label{s:summ}
%%%%%%%%%%%%%%%%%%%%%%%%%%%%%%%%%%%%%%%%%%%%%%%%%%%%%%%

In this article
we considered a general three-parameter family of regularized Navier-Stokes
and MHD
models on $n$-dimensional smooth compact Riemannian manifolds, with $n \geqs 2$;
this family captures most of the specific regularized models that have been 
proposed and analyzed in the literature.
Well-studied members of this family include the Navier-Stokes equations,
the Navier-Stokes-$\alpha$ model, the Leray-$\alpha$ model,
the Modified Leray-$\alpha$ model, the Simplified Bardina model,
the Navier-Stokes-Voight model, the Navier-Stokes-$\alpha$-like models,
and several MHD models;
the general model also captures
a number of additional models that have not been specifically 
identified or analyzed in the literature.
We gave a unified analysis of this entire family of models using essentially
only abstract mapping properties of the principal dissipation and smoothing
operators, and then used assumptions about the specific form of the 
parameterizations, leading to specific models, only when necessary to obtain 
the sharpest results.
In~\S\ref{s:prelim}, we established our notation and gave some basic 
preliminary results for the operators appearing in the general regularized 
model.
In~\S\ref{s:well}, we built some well-posedness results for the general model; 
this included existence results (\S\ref{ss:exist}), 
regularity results (\S\ref{ss:reg}),
and uniqueness and stability results (\S\ref{ss:stab}).
In~\S\ref{s:pert} we established some results for singular perturbations,
which as special cases include the inviscid limit of viscous models and the 
$\alpha \rightarrow 0$ limit in $\alpha$ models;
this involved a separate analysis of the linear (\S\ref{ss:pert-lin}) and
nonlinear (\S\ref{ss:pert-nonlin}) terms.
In~\S\ref{s:attr}, we showed existence of a global attractor for the 
general model (\S\ref{ss:attr-exist}), 
and then gave estimates for the dimension 
of the global attractor (\S\ref{ss:attr-dim}).
In~\S\ref{s:determining} we
established some results on determining operators for the
two distinct subfamilies of dissipative (\S\ref{ss:dis})
and non-dissipative (\S\ref{ss:non-dis1}) models. In~\S\ref{s:Reynolds}, we established length-scale estimates for the generalized model; this makes it possible to recast our estimates for the number of freedom of turbulent flows given in~\S\ref{s:determining} in terms of the Reynolds number.

To make the paper reasonably self-contained, 
in Appendix~\ref{s:app} we also included some supporting material on 
Gronwall-type inequalities (Appendix~\ref{ss:gronwall}),
spaces involving time (Appendix~\ref{ss:anis}),
and
Sobolev spaces (Appendix~\ref{ss:sobolev}).
In addition to establishing a number of technical results for all
models in this general three-parameter family, 
the framework we developed can recover 
most of the existing existence, regularity, uniqueness, stability, attractor
existence and dimension, and determining operator results for the well-known
specific members of this family of regularized Navier-Stokes and MHD models.
Analyzing the more abstract generalized model allows for a simpler analysis 
that helps bring out the common structure of the various models,
and also helps clarify the core common features of many of the
new and existing results.
More general MHD models can be analyzed using the framework with only
minor modifications as outlined in the text.

In ongoing work, we are extending the unified analysis presented here to establish partial regularity results for the three-parameter generalized model, in the spirit of~\cite{CaKoNi82a}.  In~\cite{KaPa02a}, it was found that for the hyper-dissipative model, there exists a solution for which the Hausdorff dimension of the space-time singular set is at most $5-4\theta$.  We would like to extend this result for our generalized equation to see the interplay between the nonlinearity, which is controlled by two parameters $\theta_1$ and $\theta_2$ and the dissipative term, which is controlled by the parameter $\theta$ in the model equations.

In~\cite{GuPr05a} the notion of suitable weak solutions for NSE was defined.  The definition introduces two parameters: a discretization scale $h$ and a large eddy scale $\veps$.   We also plan to extend this unified analysis to find the interplay between the nonlinear term and dissipative term that will satisfy the proposed list of mathematical criteria when establishing a reasonable definition of large eddy simulation (LES) models.  In~\cite{GuPr05a}, the authors mentioned some technical or fundamental difficulties when establishing the convergence of the discrete approximations of the NS-$\alpha$ model to suitable weak solutions of the NSE.  We would like to use the unified analysis to see under what conditions we can recover the relationship between the regularizing and discretization parameters that will allow the model equation to be a suitable approximation to NSE.

%% file: ack.tex
We thank the reviewers for several suggestions which improved the paper, including length-scale estimates in terms of the Reynolds number. MH was supported in part by NSF Awards~0715146, 0411723, 
and 0511766, and DOE Awards DE-FG02-05ER25707 and DE-FG02-04ER25620.
EL and GT were were supported in part by NSF Awards~0715146 and 0411723.

%% file: app.tex
%%%%%%%%%%%%%%%%%%%%%%%%%%%%%%%%%%%%%%%%%%%%%%%%%%%%%%%%%%%%%%%%%
\appendix

\section{Some key technical tools and some supporting results}
\label{s:app}
%%%%%%%%%%%%%%%%%%%%%%%%%%%%%%%%%%%%%%%%%%%%%%%%%%%%%%%%%%%%%%%%%

%%%%%%%%%%%%%%%%%%%%%%%%%%%%%%%%%%%%%%%%%%%%%%%%%%%%%%%%%%%%%%%%%
\subsection{A Gr\"onwall type inequality}
\label{ss:gronwall}
%%%%%%%%%%%%%%%%%%%%%%%%%%%%%%%%%%%%%%%%%%%%%%%%%%%%%%%%%%%%%%%%%

The following is a slight generalization of the Gr\"onwall-type inequality appeared in~\cite{FMTT83} and~\cite{JoTi91}.

\begin{lemma}
   \label{lemma:gronwall_2}
Let $T>0$ be fixed, and let $x$, $y$, and $z$ be locally integrable 
and real-valued functions on $(0,\infty)$, satisfying
\begin{equation*}
\begin{split}
\liminf_{t \rightarrow \infty}
   \int_{t}^{t+T} x(\tau) d\tau > 0,
\quad
\limsup_{t \rightarrow \infty}
    \int_{t}^{t+T} x^-(\tau) d\tau < \infty,
\quad
\lim_{t \rightarrow \infty}
     \int_{t}^{t+T} y^+(\tau) d\tau = 0,
\end{split}
\end{equation*}
where $x^-=\max\{-x,0\}$ and $y^+=\max\{y,0\}$.
If $\xi$ is an absolutely continuous non-negative function on
$(0,\infty)$, and $\xi$ satisfies the following differential inequality
$$
\xi' + x \xi \le y + y \xi^p, \qquad \text{a.e.~on}~(0,\infty),
$$
for some constant $p\in(0,1]$, then $\lim_{t \rightarrow \infty} \xi(t) = 0$.
\end{lemma}

%%%%%%%%%%%%%%%%%%%%%%%%%%%%%%%%%%%%%%%%%%%%%%%%%%%%%%%%%%%%%%%%%
\subsection{Spaces involving time}
\label{ss:anis}
%%%%%%%%%%%%%%%%%%%%%%%%%%%%%%%%%%%%%%%%%%%%%%%%%%%%%%%%%%%%%%%%%

Let us recall the following well known result.
A proof can be found in~\cite{Tema77}.

\begin{theorem}\label{t:aubin}
Let $X\hookrightarrow Y\hookrightarrow Z$ be reflexive Banach spaces, with $X\hookrightarrow Y$ compact.
Let $p,q>1$ be constants, and define
\begin{equation*}
\mathcal{Y} = \mathcal{Y}(0,T;p,q;X,Z) 
= 
\{v\in L^{p}(0,T;X):\frac{dv}{dt}\in L^{q}(0,T;Z)\}.
\end{equation*}
Then we have $\mathcal{Y}\hookrightarrow C(0,T;Z)$ and the embedding $\mathcal{Y}\hookrightarrow L^{p}(0,T;Y)$ is compact.
\end{theorem}

%%%%%%%%%%%%%%%%%%%%%%%%%%%%%%%%%%%%%%%%%%%%%%%%%%%%%%%%%%%%%%%%%
\subsection{Multiplication in Sobolev spaces}
\label{ss:sobolev}
%%%%%%%%%%%%%%%%%%%%%%%%%%%%%%%%%%%%%%%%%%%%%%%%%%%%%%%%%%%%%%%%%

With $s\in\R$, let $H^s$ be the standard Sobolev space on an $n$-dimensional compact Riemannian manifold $M$.
We state here a well known result on pointwise multiplication of functions in Sobolev spaces.

\begin{lemma}\label{l:sob-hol}
Let $s$, $s_1$, and $s_2$ be real numbers satisfying
\begin{equation*}
s_1+s_2\geqs0,
\qquad
s_i\geqs s,
\qquad
\textrm{and}
\qquad
s_1+s_2-s > \textstyle\frac{n}{2},
\end{equation*}
where the strictness of the last two inequalities can be interchanged if $s\in\N_0$.
Then, the pointwise multiplication of functions extends uniquely to a continuous bilinear map
\begin{equation*}
H^{s_1}\otimes H^{s_2}\rightarrow H^{s}.
\end{equation*}
\end{lemma}

\begin{proof}
A proof is given in~\cite{jZ77} (see also~\cite{Ma04a}) for the case $s\geq0$, and by using a duality argument one can easily extend the proof to negative values of $s$.
\end{proof}